\documentclass[12pt]{amsart} 
\usepackage[french, english]{babel}
\usepackage[T1]{fontenc}
\usepackage[utf8]{inputenc}
\usepackage[mathscr]{euscript}

\usepackage{amsmath,amssymb,latexsym,amsthm,enumerate,mathtools,marginfix,tikz}
\usepackage{hyperref}

\DeclareUnicodeCharacter{2192}{\to}
\DeclareUnicodeCharacter{3C6}{\phi}
\DeclareUnicodeCharacter{221E}{\infty}
\DeclareUnicodeCharacter{2208}{\in}
\DeclareUnicodeCharacter{2265}{\ge}
\DeclareUnicodeCharacter{3A3}{\Sigma}
 
\usepackage[abbrev, bibtex-style, nobysame]{amsrefs} 
\usepackage{xparse,etoolbox}

\BibSpec{misc}{
  +{}{\PrintAuthors} {author}
  +{,}{ \textit} {title}
  +{}{ \parenthesize} {date}
  +{,}{ \url} {url}
  +{,}{ } {note}
  +{.}{ } {transition}
}

\usepackage{tikz, tikz-cd} 
\usepackage{microtype} 
\usepackage[margin=1in]{geometry}

\usepackage{amssymb} 
\def\acts{\curvearrowright}

\renewcommand{\d}{\partial }

\DeclareMathOperator\image{im} 

\numberwithin{equation}{section}
\newtheorem{Theorem}{Theorem}[section]

\newtheorem{Prop}[Theorem]{Proposition} 
\newtheorem{Corollary}[Theorem]{Corollary} 
\newtheorem{Lemma}[Theorem]{Lemma}
\newtheorem{Definition}[Theorem]{Definition}

\theoremstyle{definition} 
\newtheorem*{overview}{Overview}
\newtheorem*{Acknowledgments}{Acknowledgments}

\newtheorem*{Question}{Question}
\newtheorem*{Assumption}{Assumption}
\newtheorem{A special case}{A special case}[Theorem]
\newtheorem*{Coarse van Kampen obstruction class}{Coarse van Kampen obstruction class}
\newtheorem*{Coarse cohomology of the configuration space}{Coarse cohomology of the configuration space}

\theoremstyle{remark}
\newtheorem{Remark}[Theorem]{Remark}
\newtheorem{Example}[Theorem]{Example}
\newtheorem*{Notation}{Notation}

\newcommand{\rr}{\mathbb{R}}
\newcommand{\zz}{\mathbb{Z}}
\newcommand{\nn}{\mathbb{N}}

\newcommand{\cU}{\mathcal{U}}

\newcommand{\csu}[1][*]{\myc_{#1}^{\cU}}

\DeclareMathOperator\size{size}
\DeclareMathOperator\diam{diam}
\DeclareMathOperator\cadim{cadim}
\DeclareMathOperator\actdim{actdim}
\DeclareMathOperator\cobdim{cobdim}

\DeclareMathOperator\myc{C}
\DeclareMathOperator\myh{H}

\DeclareMathOperator\Fix{Fix_{\mathrm G}}
\DeclareMathOperator{\Conf}{Conf}

\newcommand{\comment}[1]{}

\newcommand{\C}[1][*]{\myc^{#1}}
\newcommand{\Cas}[1][*]{\myc^{#1}_{as}}

\newcommand{\Cx}[1][*]{\myc{\mathrm X}^{#1}}
\newcommand{\cx}[1][*]{\myc{\mathrm X}_{#1}}
\newcommand{\eCx}[1][*]{\myc{\mathrm X}_{\mathrm G}^{#1}}
\newcommand{\eCb}[1][*]{\myc{\mathrm B}_{\mathrm G}^{#1}}
\newcommand{\eC}[1][*]{\myc^{#1}_{\mathrm G}}
\newcommand{\cff}[1][*]{\myc_{#1}^{F}}

\newcommand{\Cb}[1][*]{\myc{\mathrm B}^{#1}}

\newcommand{\Cz}[1][*]{\myc^{#1}_{0}}

\newcommand{\cs}[1][*]{\myc_{#1}^{\mathrm s}}
\newcommand{\clf}[1][*]{\myc_{#1}^{lf}}
\newcommand{\cf}[1][*]{\myc_{#1}}

\renewcommand{\H}[1][*]{\myh^{#1}}

\newcommand{\Hx}[1][*]{\myh{\mathrm X}^{#1}}
\newcommand{\hx}[1][*]{\myh{\mathrm X}_{#1}}

\newcommand{\h}[1][*]{\myh_{#1}}
\newcommand{\eHx}[1][*]{\myh {\mathrm X}_{\mathrm G}^{#1}}
\newcommand{\eHb}[1][*]{\myh {\mathrm B}_{\mathrm G}^{#1}}
\newcommand{\Hb}[1][*]{\myh {\mathrm B}^{#1}}

\newcommand{\eH}[1][*]{\myh^{#1}_{\mathrm G}}
\newcommand{\rH}[1][*]{\tilde{\myh}^{#1}}

\newcommand{\reH}[1][*]{\tilde{\myh}^{#1}_{\mathrm G}}

\newcommand{\gs}{\sigma }

\newcommand{\gD}{\Delta}

\renewcommand{\ge}{\varepsilon}

\let\oldsubset\subset
\renewcommand{\subset}[1][]{\overset{#1}{\oldsubset}}

\let\oldin\in
\renewcommand{\in}[1][]{\overset{#1}{\oldin}}

\let\oldnotin\notin
\renewcommand{\notin}[1][]{\overset{#1}{\oldnotin}}

\DeclarePairedDelimiter\abs{\lvert}{\rvert}
\newcommand{\Supp}[2][]{\abs{#2}_{#1}}

\DeclarePairedDelimiter\spp{\lVert}{\rVert}

\NewDocumentCommand\ccap{o}{\mathbin{\overset{\mathrm{c}
	\IfNoValueTF{#1}{} {, #1} }
	{\cap}  } }
\NewDocumentCommand\nceq{o}{\mathbin{\overset{\mathrm{c}
	\IfNoValueTF{#1}{} {, #1} }
	{\neq}  } }	
\NewDocumentCommand\cminus{o}{\mathbin{\oset{\mathrm{c}
	\IfNoValueTF{#1}{} {, #1} }
	{-}  } }	
	
\NewDocumentCommand\ceq{o}{\mathbin{\overset{\mathrm{c}
	\IfNoValueTF{#1}{} {, #1} }
	{=}  } }
\NewDocumentCommand\cneq{o}{\mathbin{\overset{\mathrm{c}
	\IfNoValueTF{#1}{} {, #1} }
	{\neq}  } }	
	
\NewDocumentCommand\csubset{o}{\subset[\mathrm{c}\IfNoValueTF{#1}{} {, #1} ]}

\makeatletter
\newcommand{\oset}[3][0ex]{
  \mathrel{\mathop{#3}\limits^{
    \vbox to#1{\\kern-2\ex@
    \hbox{$\scriptstyle#2$}\vss}}}}
\makeatother

\let\oldminus\-

\makeatletter
\newcommand{\DeclareMathActive}[2]{
  
  \expandafter\edef\csname keep@#1@code\endcsname{\mathchar\the\mathcode`#1 }
  \begingroup\lccode`~=`#1\relax
  \lowercase{\endgroup\def~}{#2}
  \AtBeginDocument{\mathcode`#1="8000 }
}

\newcommand{\std}[1]{\csname keep@#1@code\endcsname}
\patchcmd{\newmcodes@}{\mathcode`\-\relax}{\std@minuscode\relax}{}{\ddt}
\AtBeginDocument{\edef\std@minuscode{\the\mathcode`-}}
\makeatother

\begin{document}
\title{Coarse cohomology of configuration space and coarse embedding}
\author{Arka Banerjee}
\address{Department of Mathematics and Statistics\\
Auburn University\\
Auburn, AL~36849}
\email{azb0263@auburn.edu}

\begin{abstract}
    We introduce a notion of equivariant coarse cohomology of the complement of a subspace in a metric space.
      We use this cohomology to define a notion of coarse cohomology of the two-points configuration space of a metric space and develop tools to compute this cohomology under various conditions.
    As an application of this theory, we show that certain classes in the coarse cohomology of two-points configuration space  obstruct  coarse embedding between two metric spaces.
\end{abstract}

\maketitle
 \section{Introduction}
 Van Kampen~\cite{VK33} developed  an obstruction theory for embeddings of $n$-dimensional simplicial complexes into $\rr^{2n}$. 
 A modern approach to his theory uses (co)homology of the two-points configuration space. In this article, we develop an analogous obstruction theory for coarse embedding by introducing a notion of coarse cohomology of two-points configuration space.
 For simplicity, throughout this article, we will say configuration space to mean two-points configuration space. 
 
Let us first briefly describe the classical van Kampen obstruction for a topological space $X$. 
Let $\delta(X)$ be the diagonal set $\{(x,x)\mid x\in X\}\subset X\times X$. 
Consider the following deleted product
 \[\Tilde{X}:= X\times X-\delta(X)=\{(x,y)\in X\times X \mid x\neq y\}
\]
$\Tilde{X}$ has a natural free action by $\zz_2$ by switching the coordinates.
Consider the corresponding $\zz_2$ covering map $q:\Tilde{X}\rightarrow \Tilde{X}/\zz_2$.
The space $\Tilde{X}/\zz_2$ which we  denote by $\Conf(X)$, is the unordered  configuration space of two points in $X$.
There exists a classifying map from the $\zz_2$-bundle $q:\Tilde{X}\rightarrow \Tilde{X}/\zz_2$ to the universal $\zz_2$-bundle $S^\infty \rightarrow \rr P^\infty$ as follows.

\[
\begin{tikzcd}
\Tilde{X} \arrow{r}{\tilde{\phi}} \arrow{d}{q} & S^\infty \arrow{d} \\
\Conf(X) \arrow{r}{\phi} & \rr P^\infty
\end{tikzcd}
\]

If there is an embedding $g:X\hookrightarrow \rr^{n}$, then we can choose $\tilde{\phi}$ so that it factors through $S^{n-1}$. More precisely, we can choose $\tilde{\phi}$ to be the following map $\tilde{X}\rightarrow S^{n-1}\subset S^\infty$
\[(x,y)\mapsto\frac{g(x)-g(y)}{|g(x)-g(y)|}
\]
In this case, $\phi$ maps $\Conf(X)$ to $\rr P^{n-1}\subset \rr P^\infty$.

 So the induced map $\phi^*:\H[n](\rr P^\infty)\rightarrow \H[n](\Conf(X))$ is trivial as it factors through $\H[n](\rr P^{n-1})$.
 In particular, if $\eta^n \in \H[n](\rr P^\infty;\zz_2)$ denotes 
 the nonzero class, then $\phi^*(\eta^n)$ would be trivial. 
 In other words, the cohomology class $\phi^*(\eta^n)$ gives an obstruction for the embedding of  $X$ into $\rr^n$. We will call the class $\phi^*(\eta^n)$  the van Kampen obstruction class of degree $n$ and  denote it by $vk^n(X)$.

 Let us now turn our attention to the coarse world. 
 A map $f:X\rightarrow Y$ between two metric spaces is said to be a \emph{coarse embedding} if there exist two proper non-decreasing maps 
 $\rho_-,\rho_+:[0,\infty)\rightarrow [0,\infty)$ such that 
 \[\rho_-(d(x,y))\leq d(f(x),f(y)) \leq \rho_+(d(x,y)) \quad \text{for all} \quad x,y\in X
 \]
 Roe~\cite{r93} defined the notion of coarse cohomology of a metric space which can be thought of as a coarse analog of Alexander--Spanier
 cohomology in the topological setting. 
 The main motivation of this paper is to define a coarse version of $vk^*$ that obstructs coarse embedding between metric spaces.
 
The study of obstruction to the coarse embedding of finitely generated groups into proper contractible $n$-manifold was initiated by Bestvina, Kapovich, and Kleiner~\cite{BKK}.
For several interesting classes of metric spaces (for example CAT(0) space, Gromov-hyperbolic space) one can attach a boundary at infinity to the space.
Suppose $\partial_\infty X$ denotes such boundary of a space $X$.
A popular theme in the study of large scale geometry is to find topological properties of $\partial_\infty X$ that provide information about the coarse geometry of $X$.
A special case of the work of Bestvina--Kapovich--Kleiner~\cite{BKK} roughly says the following: if  $\partial_\infty X$ does not embed into $\rr^n$ due to the van Kampen embedding obstruction, then $X$ does not coarsely embed into $\rr^{n+1}$ with the Euclidean metric.
  One of the motivations behind the present work was to understand their obstruction in the language of Roe's coarse cohomology.

To this end, recall that the van Kampen obstruction class $vk^*(X)$ lives in the cohomology of the configuration space of $X$, which is the same as the $\zz_2$-equivariant cohomology of $X\times X-\delta(X)$ where the coefficients 
are $\zz_2$ with the trivial $\zz_2$ action.
This suggests that a coarse version of van Kampen obstruction class should live in some coarse version of  $\zz_2$-equivariant cohomology of complement of $\delta(X)$ in $X\times X$.
This motivates us to define a notion of equivariant coarse cohomology of the complement of subspace in a metric space.

 Building on Roe's theory, authors in~\cite{BB20} defined a notion of coarse cohomology of complement of a subspace  in a metric space.
 In this paper, we extend the work of~\cite{BB20} to the equivariant setting.
 We then define the coarse cohomology of the configuration space of $X$ simply to be the $\zz_2$-equivariant coarse cohomology of the complement of $\delta(X)$ in $X\times X$, where the action of $\zz_2$ on $X\times X$ is by switching coordinate.

 Once we have a proper notion of the coarse cohomology of the configuration space, we can search for a coarse $vk^*$ in that cohomology that obstructs coarse embeddings.
 Indeed, when $X$ is a separable metric space, we find a class in the $n^{th}$  degree of the coarse cohomology of the configuration space of $X$, which we denote by $cvk^{n}(X)$, that obstructs coarse embedding of $X$ into $\rr^{n-1}$. 
 \footnote{While the van Kampen obstruction $vk^n$ is
associated to $\rr^n$, note that the coarse van Kampen obstruction $cvk^n$ is related to $\mathbb{R}^{n-1}$. The reason for that is coarse cohomology of the configuration space of $\rr^n$ is same (except in degree 0 and 1) as cohomology of the configuration space of $\rr^{n-1}$ (cf. example~\ref{cconfR^n}).}
In general, the class $cvk^n(X)$ can be used to get obstruction to coarse embedding into any other metric space.
 We define the coarse obstruction dimension $\cobdim(X)$ of a separable metric space $X$ to be  the largest $n$ such that  $cvk^{n}(X)$ is nonzero. One of our main theorem is the following:
 
 \begin{Theorem}\label{main theorem2}
 If $X$ admits a coarse embedding into $Y$, then $\cobdim(X)\leq \cobdim(Y)$.
 \end{Theorem}
 So one way to determine when a given space does not admit coarse embedding into another space is to compare their $\cobdim$.
 However, the equivariant coarse cohomology of the complement where $cvk^*$ lives, is hard to compute in general. 
 A big part of this paper is devoted to the computation of equivariant coarse cohomology of the complement for certain spaces which may be of independent interest.
 We use these computations to estimate $\cobdim$ of certain spaces and obtain coarse embedding obstruction between certain classes of spaces.
 Below we highlight some of our results in this direction.  
 \begin{itemize}
 \item We show that 
 $\cobdim(\rr^n)=n$. Hence Theorem~\ref{main theorem2} implies that $X$ does not admit coarse embedding into $\rr^{n-1}$ if $\cobdim(X)\geq n$.

 \item Suppose  $X=K\times [0,\infty)/K\times \{0\}$ is the open cone on a finite simplicial complex $K$.
 A metric $d$ on $X$ is called expanding if for any two disjoint simplices $\gs,\tau\in K$ and $S\geq 0$, there exists $r\geq 0$ such that $d(\gs\times [r,\infty),\tau\times [r,\infty))\geq S$.
 If there is a class $c\in \h[n](\Conf(X))$ such that $vk^{n}(X)(c)\neq 0$, then we show that $\cobdim(X)\geq n+1$ whenever $X$ is equipped with a proper, expanding metric~(Example \ref{BKK obstruction}).
         Hence such $X$ does not admit coarse embedding into $\rr^n$ by Theorem~\ref{main theorem2} and the previous example. 
         This was initially proved by Bestvina--Kapovich--Kleiner~\cite{BKK}.

     \item If $X$ is a proper, uniformly acyclic $n$-manifold with uniformly locally acyclic boundary, then $\cobdim(X)\leq n$ (Theorem \ref{cobdim of manifold}).
         Examples of such spaces include universal cover of compact aspherical $n$-manifolds. 
         Hence Theorem~\ref{main theorem2} implies that, if $\cobdim(X)\geq n$, then $X$ does not admit coarse embedding into the universal cover of any compact aspherical $(n-1)$-manifold. 

         \item If $\Hx(X^2-\delta(X))=0$ for $*\leq n-1$, then $\cobdim(X)\geq n$ (Theorem~\ref{t:cobdim of unifcontr}).
     This implies, in particular, that any proper, uniformly contractible $n$-manifold has $\cobdim\geq n$.
    Hence it follows from the last example that any proper, uniformly contractible $n$-manifold does not admit coarse embedding into the universal cover of any compact aspherical $(n-1)$-manifold (see Corollary~\ref{lb of cob} for a more general version).
    This recovers a result of Yoon~\cite{Yoon}.
     
         \item If a finitely generated group $G$ acts properly on $X$ by isometries, then there exists a coarse embedding of $G$ (equipped with a word metric coming from a finite generating set) into $X$ by mapping $G$ into an orbit of the action.
 That means, from the coarse point of view, any space $X$ with a proper $G$-action has to be at least as large as $G$. 
More precisely, using Theorem~\ref{main theorem2} we show that if $G$ acts properly, cocompactly on a contractible manifold $M$ (possibly with boundary), then $\dim(M)\geq \cobdim(G)$ (Corollary \ref{cadim>cobdim}).

     \end{itemize}

Finally, we remark that equivariant coarse (co)homology has been previously studied in~\cite{Bunke} and~\cite{Wulff}. While our approach focuses specifically on metric spaces to address the obstruction theory and computations relevant to this article, the theories developed in~\cite{Bunke} and~\cite{Wulff} apply to more general spaces (referred to as coarse spaces) and are motivated by certain coarse K-theoretic considerations (as discussed in the introduction of~\cite{Wulff}). It would be intriguing to explore the extent to which our theory and computations relate to those in~\cite{Bunke} and~\cite{Wulff}.
  
  \begin{overview}  
 In section~\ref{preliminaries} we describe several variations of Alexander–Spanier cochains and the corresponding cohomology theories.
 We also introduce coarse language and define coarse (co)homology. 
 In section~\ref{s:coa of compl},
 we define the equivariant coarse cohomology of the complement and give some examples using Theorem~\ref{c:equicoa=coaquo} that relate, for certain cases,  equivariant coarse cohomology to the Alexander--Spanier cohomology of the quotient.
 In section~\ref{s:computaion of equiv} we prove Theorem~\ref{c:equicoa=coaquo}.
 In section~\ref{s:coarse configuration}, we define the coarse cohomology of the configuration space.
 In section~\ref{s:cvk}, we give a class in the coarse cohomology of the configuration space that obstructs coarse embedding maps.
 Then we define the coarse obstruction dimension of a space and prove Theorem~\ref{t:cobdim increases} which is a slightly stronger version of Theorem~\ref{main theorem2}.
 In section~\ref{relation between van kampens}, we give a relation between classical van Kampen obstruction and coarse van Kampen obstruction.
 We use this relation to compute coarse van Kampen obstruction for certain Euclidean Cones on simplicial complexes. 
 In section~\ref{s:coarse gysin}, we produce a coarse version of the Gysin sequence for the $\zz_2$-bundle to compute the coarse cohomology of configuration space.
 In section~\ref{s:ub} and~\ref{s:lb}, we use the coarse Gysin sequence to estimate $\cobdim$ of certain spaces.
 \end{overview}
\begin{Acknowledgments}
Part of this paper was written while the author was a graduate student at the University of Wisconsin–Milwaukee. The author would like to express sincere gratitude to his advisor, Boris Okun, for his valuable insights and guidance during the preparation of this work. The author also wishes to thank Kevin Schreve for helpful discussions. Finally, the author is grateful to the anonymous referee for her/his detailed and constructive comments on an earlier version of this manuscript, which significantly improved its presentation.
\end{Acknowledgments}

 \section{Preliminaries}\label{preliminaries}

\subsection{Alexander--Spanier complexes} 
We will refer to points in $X^{n+1}$ as \emph{$n$-simplices}.
 Let $R$ be an abelian group.
We will think of functions $X^{n+1} \to R$ as $n$-cochains or $n$-chains on $X$, depending on the context, and in the latter case will use  additive notation $c=\sum_{\gs\in X^{*+1}}c_\gs \gs$.

The basic complex is the complex of finitely supported chains
\[
\cf(X;R):=\{c=\sum_{i=0}^n c_i \sigma_i \mid c_i \in R \text{ and } \sigma_i\in X^{*+1}\}
\]
equipped with the usual  boundary map, defined on the basis by 
\[
\partial (x_0,\cdots,x_n):=\sum_{i=0}^n (-1)^i(x_0,\dots,\hat{x}_i,\dots,x_n).
\]

Note that the boundary map $\partial$ is well-defined on a larger complex of \emph{locally finite} chains $\clf(X; R)$ which consists of chains $c$ satisfying the following property: for any bounded $B \subset X$ only finitely many simplices in $c$ have vertices in $B$.

The algebraic dual of  $\cf(X)$ is the complex of all Alexander--Spanier cochains
\[
\C[*](X;R)=\{\phi: X^{*+1} \to R\}
\]
with the coboundary operator  
\[
d(\phi)(x_0,\dots,x_n)=\sum_{i=0}^n(-1)^{i} \phi(x_0,\dots,\hat{x}_i,\dots,x_n).
\]
\begin{Lemma}\label{l:C is acyclic}
    $\C(X;R)$ is an acyclic complex i.e the homology of the complex is trivial in each degree except at degree zero where it is isomorphic to $R$.
\end{Lemma}
\begin{proof}
In degree zero, the cohomology consists of all the constant functions $X\rightarrow R$ which is isomorphic to $R$.

  To prove that the cohomology is trivial in degree $\geq 1$,   choose $a\in X$. 
Consider the following cone operator $D_*:\cf(X)\rightarrow \cf[*+1](X)$ 
\begin{equation*}\label{Coneoperator}
   D_n:(x_0,x_1,\ldots,x_n)\mapsto (a,x_0,\ldots,x_n). 
\end{equation*} 
For any $n$-simplex $\gs$ with $n\geq 1$, we observe that $D_*\partial(\gs)+\partial D_*(\gs)=\gs$.
By considering the dual, we have that for any cochain $\phi\in \C(X;R)$ in degree $\geq 1$,
\[
dD^*(\phi)+D^*d(\phi)=\phi.
\]
In particular,
$dD^*(\phi)=\phi$ when $\phi$ is a cocycle.
The claim follows.
\end{proof}

All the cochain complexes that we will consider in this paper would be subcomplexes 
of the complex $(\C,d)$.

For a cochain $\phi\in \C(X)$, let $||\phi||$ be the intersection of the diagonal $\gD=\{(x,x,\ldots, x)\mid x\in X\}\subset X^{*+1}$ and the closure of the support of the function $\phi:X^{*+1}\rightarrow R$.
Let $\Cz(X;R)$ be the complex of locally zero cochains:
\[
\Cz(X;R) = \{\phi \in \C(X;R) \mid \spp{\phi} =\emptyset \}.
\]
 Note that the restriction of $d$ gives a well defined map $\C_0(X)\rightarrow \C[*+1]_0(X)$. 
 Consequently, $d$ induces a well defined map $\Cas(X;R)\rightarrow \Cas[*+1](X;R)$, where
 \[\Cas(X;R)=\C(X;R)/\Cz(X;R).\] 
Alexander--Spanier cohomology, denoted by $\H(X;R)$, is the cohomology of the complex $(\Cas(X;R),d)$.
We will denote by $\rH(X;R)$ the reduced Alexander--Spanier cohomology.

\subsection{Coarse inclusion}We adopt the notation from~\cite{msw11}. Let $(X,d)$ be a metric space.
For $A \subset X$ and $r\geq 0$, we denote $N_{r}(A)=\{x \in X \mid d(x, A) \leq r \}$.
We will call such neighborhoods \emph{metric neighborhoods} of $A$.
We will say that $A$ is \emph{$r$-contained} in $B$, $A \subset[r] B$, if $A \subset N_{r}(B)$.
$A$ is \emph{coarsely contained} in $B$, $A \csubset B$, if $A \subset[r] B$ for some $r\geq 0$.
Two subsets are \emph{coarsely equal}, $A \ceq B$ if $A \csubset B$ and $B \csubset A$.

\subsection{Coarse intersection} Now we recall from~\cite{msw11} the notion of \emph{coarse intersection}: $A \ccap B \ceq C$ if for all sufficiently large $r\geq 0$, $N_{r}(A) \cap N_{r}(B) \ceq C$.
The coarse intersection is not always well-defined, it may happen that the coarse type of $N_{r}(A) \cap N_{r}(B) $ does not stabilize as $r$ goes to infinity.
However the notion ``coarse intersection is coarsely contained in'' is well-defined.
$A \ccap B \csubset C$ means that for any $r\geq 0$, $N_{r}(A) \cap N_{r}(B) \csubset C$.
It is not hard to see that this condition is equivalent to the condition that for any $r\geq 0$, $A \cap N_{r}(B) \csubset C$.

\begin{Notation}
   From now on, all spaces will be assumed to be metric spaces unless stated otherwise. We will use the letter 
$r$ to represent a non-negative real number, while 
$R$ will denote an abelian group unless specified otherwise
\end{Notation}
\subsection{Coarse (co)homology}
In what follows, we will need to measure distances between simplices of different dimensions.
A convenient way to do this is to stabilize simplices by repeating the last coordinate, as follows.
Denote by $X^\infty$ the subset of the product of countably many copies of $X$, consisting of eventually constant sequences.  
Equip $X^{\infty}$ with the $\sup$ metric.
Let $i:X^{n+1} \to X^\infty$ denote the map $ (x_{0},\dots,x_{n}) \mapsto (x_{0},\dots,x_{n}, x_{n},x_{n},\dots)$.
For a function  $\phi:X^{n+1} \to R$ define its stabilized support
\[
\Supp{\phi} = \{i(\gs) \mid \gs \in X^{n+1} \text{ and } \phi(\gs) \neq 0\} \subset X^{\infty}.
\]
Let $\gD=i(X)$ denote the diagonal of $X^{\infty}$.
Define the support of $\phi$ at scale $r$ to be $\Supp[r]{\phi} = \Supp{\phi} \cap N_{r}(\gD)$.

We now define coarse (co)homology theories, following Roe~\cite{r93} and Hair~\cite{h10} using our language.
The coarse cochain complex is
\[
\Cx(X;R) := \{\phi \in \C(X;R) \mid \Supp[r]{\phi}\text{ is bounded for all }r \}.
\]
An equivalent way to define coarse cochains is to require the support to be coarsely disjoint from the diagonal:
\[
\Cx(X;R) = \{\phi \in \C(X;R) \mid \Supp{\phi} \ccap \gD \ceq * \}.
\]
The coboundary operator $d$ preserves this  property, and 
the coarse cohomology $\Hx(X;R)$ is defined to be the cohomology of  the complex $(\Cx(X;R),d)$. The coarse homology\footnote{We will need coarse homology only in Section~\ref{s:ub} for a brief discussion of coarse PD($n$) spaces. For the rest of the paper, we will work with coarse cohomology.} $\hx(X; G)$ is the homology of the following subcomplex of $\clf(X; G)$
\[
\cx(X; R):=\{c\in \clf(X; R)\mid \Supp{c}\csubset \gD\}
\]
equipped with the restriction of the boundary operator $\partial$.
Note that in the presence of the support condition $\Supp{c}\csubset \gD$ local finiteness is equivalent to $\Supp{c}$ having finite intersections with bounded subsets of $X^{\infty}$.
\begin{Example}
    Roe~\cite{r93} showed that the coarse cohomology is isomorphic to the compactly supported Alexander--Spanier cohomology if the space is uniformly contractible.
    In particular, this applies to the universal cover of finite aspherical complexes.
    In this case, the coarse homology is isomorphic to the locally finite homology~\cite{r93}*{Chapter 2}.
    For example,
    \[
    \hx(\rr^n; R)=\Hx(\rr^n; R)=
    \begin{cases}
        R & *=n, \\
        0 & \text{otherwise}.
    \end{cases}
    \]
\end{Example}

%%%%%%%%%%%%%%%%%%%%%%%%%%%%%%%
\section{Equivariant coarse cohomology of the complement}\label{s:coa of compl}
 
  We start by briefly reviewing the notion of coarse cohomology of the complement from~\cite{BB20}.
 Roughly the idea is that the coarse complement of a subset $A$ in $X$ is determined by the collection of subsets $\mathcal{S}$ of $X$ which are \textit{coarsely disjoint} from $A$:
\[\mathcal{S}:=\{B\subset X \mid N_r(B)\cap N_r(A) \text{ is bounded for any $r\geq 0$} \}
\]
Coarse cohomology of the complement of $A$, denoted by $\Hx(X-A)$,  is defined to be the cohomology of the following complex with the usual coboundary operator $d$ mentioned in the previous section. 
\[	\Cx[n](X-A;R)= \{ \phi \in \C[n](X;R) \mid  \phi|_{B}\in \Cx[n](B;R) \text{ for all } B \in \mathcal{S}\}
\]

Recall that $i:X\to X^\infty$ denote the map $x\mapsto (x, \dots, x, \dots)$ and $\gD=i(X)$.
For a subset $A\subset X$, we denote the set $i(A)\subset \gD$ by $\gD_A$.
For our purpose, we will work with the following equivalent description of $\Cx(X-A)$ (see~\cite{BB20}, Lemma 3.2)
\[
\Cx[n](X-A;R)= \{ \phi \in \C[n](X;R) \mid \Supp{\phi} \ccap \gD \csubset \gD_A \}
\]
Note that this description of coarse cochains of the complement is  closer to the spirit of Roe's original definition of coarse cochains discussed in the previous section.
\begin{Example}
    When $X\ceq A$, we can see that $\Cx(X-A;R)$ coincides with $\C(X;R)$ because in this case $\gD\csubset \gD_A$ and therefore the support condition $\Supp{\phi} \ccap \gD \csubset \gD_A$ holds for any $\phi\in \C(X;R)$. 
    Since $\C(X;R)$ is an acyclic complex,
     in this case $\Hx(X-A;R)$ is trivial in all degrees except at degree $0$ where it is isomorphic to $R$.
\end{Example}
\begin{Example}
    If $A$ is bounded, then $\Cx(X-A;R)$ coincides with $\Cx(X;R)$.
    Hence, $\Hx(X-A,R)=\Hx(X;R)$ whenever $A$ is bounded.
\end{Example}
\begin{Example} 
As explained in~\cite{BB20}*{Theorem 5.9}, the elements of $\Hx[1](X-A;\zz_2)$ correspond to the `coarse complementary components' of $A$ inside $X$.
More precisely, if $X$ is a geodesic space and $A\subset X$ and $k=\dim_{\zz/2} \Hx[1](X-A; \zz/2)$ is finite, then
     there exists $r\geq 0$ such that for any $L\geq r$, $X-N_L(A)$ has exactly $k+1$ deep path components,
    where deep path components are those path components which are not coarsely contained in $A$.  
    \end{Example}
\subsection{Equivariant coarse cohomology of the complement} We now define an equivariant version of the coarse cohomology of the complement. 
Let $G$ be a group acting on $X$ by isometries and  $R$ be an abelian group with a $G$-action. 
Suppose $G$ acts on $X^{n+1}$ by the diagonal action, $g (x_0,\dots,x_n):=(g x_0,\dots, g x_n)$.
$G$-equivariant coarse cohomology of the complement of $A$ in $X$, denoted by $\eHx(X-A;R)$, is defined to be the cohomology of the following cochain complex  
\[ 
\eCx(X-A;R):= \{\phi\in \Cx(X-A;R) \mid \text{$\phi$ is $G$-equivariant}\}
\]
with the usual coboundary operator $d$.

In particular, if $G$ is acting trivially on $X$ and $R$, then $\eCx(X-A;R)$ coincides with $\Cx(X-A;R)$ and therefore $\eHx(X-A;R)=\Hx(X-A;R)$.
While $\eHx$ is hard to compute in general,
 we can relate $\eHx$ to a  more computable cohomology under certain acyclicity conditions which we define below.
 \begin{Definition}
$X$ is locally acyclic with coefficients in $R$ if for any $x\in X$ and a neighborhood $U$ of $x$, there exists another open neighborhood $V\subset U$ of $x$ such that the inclusion $V\hookrightarrow U$ induces trivial map in the singular homology  with coefficients in $R$. 

 $X$ is called locally acyclic away from $A$ with coefficients in $R$  if for some $r$, $X-N_r(A)$ is locally acyclic with coefficients in $R$.
\end{Definition}
 \begin{Example}
 Any locally finite simplicial complex is locally acyclic.
     If a space is locally acyclic, then it is locally acyclic away from any of its subset.
     However, the converse is not true.
     For instance, any bounded metric space is locally acyclic away from any subset.
 \end{Example}

 \begin{Definition}
 $X$ is uniformly acyclic with coefficients in $R$ if there exists a non-decreasing function $\rho:\rr_{\geq 0}\rightarrow \rr_{\geq 0}$ such that for any $B\subset X$ the inclusion  map $B\hookrightarrow N_{\rho(\diam(B))}(B)$ induces trivial map in the singular homology with coefficients in $R$.
 
 $X$ is called uniformly acyclic away from $A\subset X$ with coefficients in $R$ if there exist two non-decreasing functions $\rho,\mu:\rr_{\geq0}\rightarrow \rr_{\geq 0}$  such that for any $B\subset X$ the inclusion  map $B\hookrightarrow N_{\rho(\diam(B))}(B)$ induces trivial map in the singular homology with  coefficients in $R$ if $d(B,A)\geq \mu(r)$.
 \end{Definition}
\begin{Example}
 Universal covers of compact aspherical complexes are uniformly acyclic. Any uniformly acyclic space is uniformly acyclic away from any of its subset.
 Moreover, if we remove a bounded set from a uniformly acyclic space, then the resulting space is uniformly acyclic away from any point.
 In general, uniform acyclicity away from a point can be very far from being uniformly acyclic.
 Figure~\ref{unif at inf eg} describes such an example.
\end{Example}

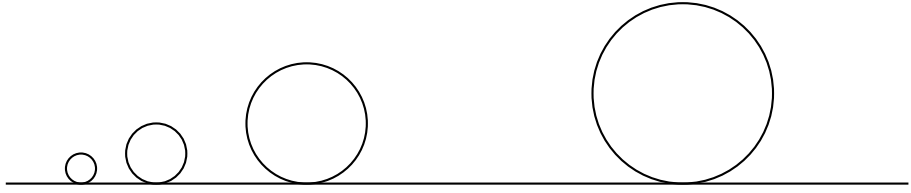
\begin{figure}\label{unif at inf eg}
\centering
\begin{tikzpicture}
    \draw[thick] (0,0)--(12,0);
    \draw[thick] (1,.2) circle(.2cm);
    \draw[thick] (2,.4) circle(.4cm);
    \draw[thick] (4,.8) circle(.8cm);
    \draw[thick] (9, 1.2) circle(1.2cm);
\end{tikzpicture}
 \caption{ A  subspace of $\rr^2$ that consists of a countable union of circles $\{C_i\}_{i\in\nn}$ and the ray $\mathcal{R}:=[0,\infty)\times \{0\}$ such that the $i^{th}$ circle has radius $i$ and touches $\mathcal{R}$ at $(i^2,0)$ so that the distance between two consecutive  circles grows to infinity. This space is uniformly acyclic away from any bounded set.}
    \label{fig:pack of circles}
\end{figure}

 For the rest of the paper $G\acts X$ will mean that $G$ is acting on $X$ by isometries and $G\acts(X,A)$ will mean  $G\acts X$ and $GA=A$.
 We let 
 \[\Fix(X):=\{x\in X\mid gx=x \text{ for all } g\in G\}.\] 
We now state a theorem that relates the equivariant coarse cohomology of the complement and the reduced Alexander--Spanier cohomology for certain spaces.

\begin{Theorem}\label{c:equicoa=coaquo}
 Suppose  $G\acts (X,A)$ and $\Fix(X)\neq \emptyset$.
Let $R$ be an abelian group with trivial $G$-action.     Suppose $X$ is uniformly acyclic away from $A\subset X$ with coefficients in $R$ and  locally acyclic away from $A$ with coefficients in $R$. 
     \begin{enumerate}
     \item \label{i:equicoa=coaquo1} If $X\csubset A$, then
     \begin{align*}
         \eHx(X-A;R)=\begin{cases}
         R & \text{if $*=0$},\\
         0 & \text{otherwise}
         \end{cases} 
     \end{align*}
     \item \label{i:equicoa=coaquo2} If $X$ is not coarsely contained in $A$, then
     \begin{align*}
         \eHx(X-A;R)=\begin{cases}
         0 & \text{if $*=0$},\\
         \varinjlim \rH[*-1]((X-N_r(A))/G;R) & \text{otherwise}
         \end{cases} 
     \end{align*}
      where  ${\rH}(-)$ is the reduced Alexander--Spanier cohomology.    
     \end{enumerate}
\end{Theorem}
We will postpone the proof of Theorem\ref{c:equicoa=coaquo} until the next section. We conclude this section with an example.

\begin{Example}\label{conf(R^n)}
    Consider the action of $\zz_2$ on $\rr^{n}$ by antipodal map. Let $M$ be a codimension $k$ vector subspace. Then for any abelian group $R$ with a trivial action of $\zz_2$ and $i\geq 1$, we have
     \begin{align*}
         \eHx[i](\rr^n-M; R) &= \varinjlim \rH[i-1]((\rr^n-N_r(M))/\zz_2;R)\quad [\text{by Theorem } \ref{c:equicoa=coaquo}] \\
         &=\rH[i-1]((\rr^{n}-M)/\zz_2; R)\\
         &=\rH[i-1](\rr P^{k-1};R)
     \end{align*}  
  The second equality holds because $\rr^n-N_r(M)$ is $\zz_2$-equivariantly homotopic to $\rr^n-M$ and the last equality follows because $\rr^{n}-M$ is $\zz_2$-equivariantly homotopic to $S^{k-1}$ with the antipodal  $\zz_2$-action. 
\end{Example}
    
\section{Computation of equivariant coarse cohomology}\label{s:computaion of equiv}
In this section, our main goal is to prove Theorem~\ref{c:equicoa=coaquo}.
The proof is similar to the proof of a non-equivariant version of Theorem~\ref{c:equicoa=coaquo} proved in~\cite{B24}*{Corollary 3.5}.
The key is to relate the coarse cohomology of complement to the boundedly supported cohomology of the complement which we introduce next.
\subsection{Boundedly supported cohomology of the complement}
Recall that $i:X\to X^\infty$ denotes the map $x\mapsto (x,\ldots,x,\ldots)$ and $\gD=i(X)$,
For a subset $A\subset X$, we denote the set $i(A)\subset \gD$ by $\gD_A$.
Also recall that for a cochain $\phi\in \C(X)$,  $||\phi||$ denotes the intersection of the diagonal $\gD=\{(x,x,\ldots, x)\mid x\in X\}\subset X^{*+1}$ and the closure of the support of the function $\phi:X^{*+1}\rightarrow R$.
Boundedly supported cohomology of the complement of $A\subset X$, denoted by $\Hb(X-A)$, is the cohomology of the following cochain complex with $d$ being the coboundary operator.
\[
\Cb(X-A;R):=\{\phi \in \C(X;R)\mid \spp{\phi} \csubset \gD_A \}
\]

Suppose $R$ is an abelian group with a $G$-action. We define the equivariant boundedly supported cohomology of the complement with coefficients in $R$ to be the cohomology of the following  cochain complex.
\[
	\eCb(X-A;R):=\{\phi \in \Cb(X-A;R) \mid \phi \text{ is $G$-equivariant}\}
\]
We denote the cohomology of the above complex by $\eHb(X-A;R)$. 

Our next goal is to relate $\eHb$ to the equivariant Alexander--Spanier cohomology.
We start by briefly recalling the notion of equivariant 
Alexander--Spanier cohomology.
\subsection{Equivariant Alexander-Spanier cohomology }Honkasalo defined a notion of equivariant Alexander--Spanier cohomology in~\cite{honkasalo}. 
Let $R$ be an abelian group with a $G$-action.
Consider the following cochian complex
\begin{equation*}\label{equivariant cochain}
    \mathscr{C}_G^*(X;R):=\{\phi\in \C(X;R) \mid \phi \text{ is $G$ equivariant}\}\hfill \tag{$\dagger$}
 \end{equation*}   
    The equivariant Alexander--Spanier cochain complex is defined as follows
\[\eC(X;R):=\mathscr{C}_G^*(X;R)/\mathscr{C}_G^*(X;R)\cap \Cz(X;R)
\]
with the usual coboundary operator $d$.
We will denote the cohomology of this complex by $\eH(X;R)$.

The following theorem of Honkasalo relates the $G$-equivariant Alexander--Spanier cohomology with the Alexander--Spanier cohomology of the quotient by the $G$-action.

\begin{Theorem}[cf. Corollary 6.8~\cite{honkasalo}]\label{equi=ord of quotient}
Let $R$ be an abelian group with the trivial $G$-action. There is a natural isomorphism 
\[\eH(X;R) \cong {\H}(X/G;R)
\] 
where the right-hand side is the ordinary Alexander--Spanier cohomology of $X/G$.
\end{Theorem}

\begin{Remark}
    To define $\eC$ in general, one needs a contravariant coefficient system---a contravariant functor from the category of $G$-spaces $G/H$ $(H\leq G)$ and $G$-maps between them to the category of abelian groups. 
 If $R$ is an abelian group with a $G$-action, it defines a contravariant coefficient system $G/H\mapsto R$, each $G$-map $G/H\rightarrow G/K$ inducing identity $R\rightarrow R$.
 With this coefficient system, the equivariant Alexander--Spanier cochain complex takes the form of $\eC(X;R)$ discussed above.
 \end{Remark}

 \subsection{Relation between $\eHb$ and $\eH$} 
 To relate the equivariant boundedly supported  cohomology with the equivariant Alexander--Spanier cohomology, we will use 
  the cohomology of the cochain complex $\mathscr{C}_G^*(X;R)$ as in~(\ref{equivariant cochain}). We will denote this cohomology by $\mathcal{H}_G^*(X;R)$.
We start with the following lemma.

\begin{Lemma}\label{l:basic equiv complex}
Suppose $G\acts X$ and $\Fix(X)\neq \emptyset$. 
Then 
\[\mathcal{H}^*_G(X;R)=\begin{cases} R & \text{if $*=0$}\\
0 & \text{otherwise.}
\end{cases}
\]
\end{Lemma}
\begin{proof}
    The proof is basically the equivariant version of Lemma~\ref{l:C is acyclic}.
$\mathcal{H}^0_G(X;R)$ consists of all the constant functions from $X$ to $R$, and hence $\mathcal{H}^0_G(X;R)=R$.
    
To prove that the cohomology is trivial in degree $\geq 1$,  choose $a\in \Fix(X)$. 
Consider the following $G$-equivariant cone operator $D_*:\cf(X)\rightarrow \cf[*+1](X)$ 
\begin{equation*}
   D_n:(x_0,x_1,\ldots,x_n)\mapsto (a,x_0,\ldots,x_n). 
\end{equation*} 
Proceeding similar to the proof of Lemma~\ref{l:C is acyclic}, we obtain that
$dD^*(\phi)=\phi$ when $\phi$ is a cocycle where $D^*$ is the dual of $D_*$.
Since $D_*$ is $G$-equivariant, we have $D^*(\phi)\in \mathscr{C}_G^{*+1}(X;R)$.
Hence, $\mathcal{H}^*_G(X)=0$ for $*\geq 1$.
\end{proof}
 We now state our main proposition that relates $\eHb$ and $\reH$ under the assumption that $\Fix(X)\neq \emptyset$.

\begin{Prop}\label{p:bdd=singular2}
Suppose $G\acts (X,A)$ and $\Fix(X)\neq \emptyset$.
\begin{enumerate}
\item\label{i:bdd=sing1} If $X\csubset A$, then \[\eHb(X-A;R)=\begin{cases}
 R & \text{if $*=0$},\\ 0 & \text{otherwise.}
\end{cases}\]
    \item\label{i:bdd=sing2} If $X$ is not coarsely contained in $A$, then \[\eHb(X-A; R)=\begin{cases} 0 & \text{if  $*=0$,}\\   \varinjlim \rH[*-1]_G(X-N_r(A); R) & \text{otherwise.}
    \end{cases}\]
    
\end{enumerate}
\end{Prop}

\begin{proof}

   \eqref{i:bdd=sing1}
      If $X\csubset A$, then $\eCb(X-A;R)$ contain all $G$-equivariant cochains on $X$, in other words, $\eCb(X-A;R)=\mathscr{C}_G^*(X;R)$.
    The claim now follows from Lemma~\ref{l:basic equiv complex}.

 \eqref{i:bdd=sing2}
 Elements in $\eHb[0](X-A;R)$ are constant functions on $X$ with support contained in $N_r(A)$ for some $r$.
 Therefore, $\eHb[0](X-A;R)=0$ if $X$ is not coarsely contained in $A$.
 
 To calculate $\eHb(X-A;R)$ for $*\geq 1$, we first observe that the following is a  short exact sequence of $G$-equivariant cochain complexes,
\begin{align*}\label{se}
0\rightarrow \eCb(X-A;R) \xrightarrow{j} \mathscr{C}_G^*(X;R) \xrightarrow{i} \varinjlim \eC(X- N_r(A);R) \rightarrow 0 \hfill \tag{$\star$}
\end{align*}
where $j$ is the inclusion map and $i$ is induced by the composition of the following canonical maps for each $r$
\begin{align*}
\mathscr{C}_G^*(X;R)\xrightarrow{restriction} \mathscr{C}_G^*(X-N_r(A);R)\xrightarrow{quotient} \C_G(X-N_r(A);R)
\end{align*}
To prove~(\ref{se}) is a short exact sequence, we observe the following:
\begin{align*}\ker(i)&=\{\phi\in \mathscr{C}_G^*(X;R)\mid \phi \in \Cz( X- N_r(A);R) \text{ for some r}\}\\
&=\{\phi\in \mathscr{C}_G^*(X;R)\mid |\phi| \cap \gD\subset N_r(\gD_A) \text{ for some $r$}\}\\
&=\{\phi\in \mathscr{C}_G^*(X;R)\mid ||\phi||\csubset \gD_A \}\\
&= \image(j).
\end{align*}
The short exact sequence~(\ref{se}) induces a long exact sequence of corresponding reduced cohomologies.
The reduced cohomology of the middle cochain complex $\mathscr{C}_G^*(X;R)$ is trivial in all degrees by Lemma~\ref{l:basic equiv complex}.
Hence, the long exact sequence implies that \[\eHb(X-A;R) \cong \varinjlim \reH[*-1](X-N_r(A);R) \quad \text{ for } *\geq 1.
\]

\end{proof}

Combining Proposition~\ref{p:bdd=singular2} and Theorem~\ref{equi=ord of quotient} we get the following.

\begin{Corollary}
    Suppose $X$ is not coarsely contained in $A$ for some $A\subset X$, $G\acts (X,A)$ and $\Fix(X)\neq \emptyset$. 
    Then
     \begin{align*}
         \eHb(X-A;R)=\begin{cases}
         0 & \text{if $*=0$},\\
         \varinjlim \rH[*-1]((X-N_r(A))/G;R) & \text{otherwise}
         \end{cases} 
     \end{align*}
\end{Corollary}

\subsection{Relation between $\eHx$ and $\eHb$} 
The main reason for defining $\eHb(X-A;R)$ is the following theorem. 
 
 \begin{Theorem}\label{t:coa=end 1}
	Suppose $G\acts (X,A)$ and $R$ is an abelian group with a $G$-action. Suppose $X$ is uniformly acyclic away from $A$, and is locally acyclic away from $A$ with coefficients in $R$. Then the inclusion $\eCx(X-A;R) \hookrightarrow \eCb(X-A;R)$ induces an isomorphism on the cohomology:
	\[
		\eHx(X-A;R) \cong \eHb(X-A;R).
	\]
\end{Theorem}
 
 % The proof of the above theorem follows a sketch that Roe described for a proof of a similar statement in~\cite{r03} which relates coarse cohomology with the compactly supported cohomology.
The proof of the above theorem follows a sketch provided by Roe in~\cite{r03}, which outlines a proof of a related result connecting coarse cohomology with compactly supported cohomology. 
The core idea of the proof involves a standard 'connect the dots' construction, which allows us to attach a singular chain to each simplex while remaining close to it, thanks to the uniform acyclicity of the space. 
 We will state this 
 more precisely in Lemma~\ref{l:filling}.

Let us first fix some notation. For the rest of the section, we suppress the coefficient $R$ from the notation unless it is important.
Let  $\cs(X)$ be the singular chain complex on $X$.

For the upcoming lemmas, we need to measure the distance between the support of a singular simplex and an $n$-simplex. Note that support of a singular simplex is in $X$ while an $n$-simplex is in $X^{n+1}$. However, recall from Section~\ref{preliminaries} that $X^n$ can be realized as a subset of $X^\infty$ for any $n\in \nn$ by stabilizing the last coordinate. 
This way both the support of a singular chain, and more generally any subset of $X$ and  $n$-simplices  live in $X^\infty$.
Consequently, we can measure the distance between the support of a singular simplex and an $n$-simplex in $X^\infty$.

In what follows, we will need a bound on the distance between a simplex and it's boundary. This is the purpose of the following lemma.
% \textcolor{red}{added the lemma}
\begin{Lemma}\label{l:d(simplex,boundary)}
$\max_{\tau\in |\d\gs|}d(\gs,\tau)\leq \diam(\gs)$ for any simplex $\gs$.
\end{Lemma}
\begin{proof}
Let $\gs=(x_0,x_1,\ldots,x_n)$ and $\tau\in |\d \gs|$. Note that the stabilization of both $\gs$ and $\tau$ in $X^\infty$ has $x_i$ in each coordinate for some $i$.
It follows that
\[d(\gs,\tau)\leq \max_{i\neq j}\{|x_i,x_j|\}= \diam(\gs).\]
\end{proof}
% \textcolor{red}{extended the proof}

\begin{Lemma}\label{l:filling}
    Suppose $X$ is uniformly acyclic away from $A$ and locally acyclic away from $A$ and $G\acts (X,A)$. Then there exist two non-decreasing sequences of functions $\mu_n,\rho_n:[0,\infty)\rightarrow [0,\infty)$ and
 a $G$-equivariant chain map $M: \cff(X) \to \cs(X)$ where 
 \[
\cff[n](X)=\langle \gs^{n} \mid d(\gs^{n},A) \geq  \mu_{n}(\diam \gs^{n}) \rangle \subset \cf[n](X)
\]
 such that
\begin{enumerate}
 \item \label{i:filling 1}$\Supp{M(\gs^{n})} \subset N_{\rho_{n}(\diam(\gs^n))}( \gs^{n})$.
 \item \label{i:filling 2} There exists an $L\geq 0$ such that for every $x\in X- N_L(A)$, and a neighborhood $U$ of $x$, there is a neighborhood $W\subset U$ of $x$ such that $\Supp{M(\gs^n)}\subset U$ for all $\gs^n\in {W}^{n+1}$.
 \end{enumerate}
 \end{Lemma}

\begin{proof}
    We will define $\mu_i, \rho_i$, and the chain map $M:\cff[i]\rightarrow \cs[i]$ by induction on $i$.
    
    For $i=0$, define $\mu_0$ and $\rho_0$ to be the constant functions $[0,\infty)\rightarrow [0,\infty)$ with image $\{0\}$.
    That means, $\cff[0](X)=\cf[0](X)$ and
    we define $M:\cff[0](X)\rightarrow \cs[0](X)$ to be the identity map.

   Since $X$ is uniformly acyclic away from $A$, there exists
   $\rho,\mu:\rr_{\geq0}\rightarrow \rr_{\geq 0}$  such that the map $i_*:\h(B)\rightarrow \h(N_{\rho(\diam(B))}(B))$ is trivial if $d(B,A)\geq \mu(\diam(B))$.   
   By the induction hypothesis, suppose $\mu_i, \rho_i$, and $M:\cff[i]\rightarrow \cs[i]$ are already defined with the desired properties for all $i\leq n$.
  For $i=n+1$, define 
  \begin{align}\label{d1}
      \mu_{n+1}(x)=\max\{\mu[2\rho_n(x)+x]+\rho_n(x)+x, \mu_n(x)\}
    \end{align}
 By construction, $\mu_{n+1}\geq \mu_n$.

  Next, we define $M$ on $\cff[n+1](X)$.
   Take $\gs \in \cff[n+1](X)$ with $\diam(\gs)=r$.
   By the induction hypothesis, $M$ satisfies property~\eqref{i:filling 1} when applied to $\d \gs$ and therefore
   \begin{align}\label{d2}
        |M(\d \gs)|)\leq N_{\rho_n(r)}(|\d\gs|)  
\end{align}
  
Consequently we obtain 
   \begin{align}\label{inequality 1}
       d(\gs,|M(\d\gs)|)\leq d(|\d\gs|,|M(\d\gs)|)+d(\gs,|\d\gs|)\leq \rho_n(r)+r
\end{align}
   where the first inequality is due to triangle inequality and the second inequality follows from \ref{d2} and Lemma~\ref{l:d(simplex,boundary)}.
   
   Since $\gs\in \cff[n+1](X)$, we have $d(\gs,A)\geq \mu_{n+1}(r)$.
   It follows that
   \begin{align*}
       d(|M(\d \gs)|,A)&\geq d(\gs,A)-d(\gs, |M(\d \gs)|) \quad &\text{[by triangle inequality]}\\
       &\geq \mu_{n+1}(r)-\rho_n(r)-r &[\text{by \ref{inequality 1}}]\\
       &\geq \mu(2\rho_n(r)+r)+\rho_n(r)+r-\rho_n(r)-r &[\text{by \ref{d1}}]\\
       &=\mu(2\rho_n(r)+r)
   \end{align*}
   Note that, $|M(\d \gs)|)\subset  N_{\rho_n(r)}(|\d\gs|)$ and $\diam(|d\gs|)\leq r$ implies that  
   \[\diam(|M(\d\gs)|)\leq 2\rho_n(r)+r
   \]
   Since $d(|M(\d \gs)|,A)\geq \mu(2\rho_n(r)+r)$,
   and $X$ is  $(\mu,\rho)$-uniformly acyclic away from $A$,  it follows that $M(\d\gs)$ is a boundary of some singular chain of diameter at most $\rho[2\mu_n(r)+r]$.
   Let 
   \[k(\gs):=\inf\{\diam(c)\mid \d c=M(\d \gs)\}. 
   \]
   We define $M(\gs)$ to be a singular chain whose boundary is $M(\d\gs)$ and has diameter at most $2k(\gs)$.
   To make $M$ a $G$-equivariant chain, we can first define $M$ on a set of simplices from $\cff[n+1](M)$ that contains one element from each orbit of simplices under the action of $G$ and then extend the map $G$-equivariantly.
   
   Next we define $\rho_{n+1}$, so that $M$ satisfies property~\ref{prop1}.
   Let $\gs$ be as before.
   By construction, $\d M(\gs)=M(\d \gs)$ and $k(\gs)\leq \rho[2\mu_n(r)+r]$.
   Therefore we have
   \begin{align*}
       |M(\gs)|\subset N_{2k(\gs)}|M(\d\gs)|&\subset N_{2k(\gs)+\rho_n(r)}(|\d \gs|)\quad && \text{[by \ref{d2}]} \\
       &\subset N_{2k(\gs)+\rho_n(r)+r}(\gs)&& \text{[by \ref{l:d(simplex,boundary)}]}\\
       &\subset N_{2\rho[(2\mu_n(r)+r)]+\rho_n(r)+r}(\gs)
       \end{align*}
Finally if we define 
   \[\rho_{n+1}(x):=2\rho[2\mu_n(x)+x]+\rho_n(x)+x,
   \]
then by construction
\[M(\gs)\subset N_{\rho_{n+1}(r)}(\gs)= N_{\rho_{n+1}(\diam(\gs))}(\gs)
\]
which is the  desired property~\ref{prop1}.

   To prove~\eqref{i:filling 2}, recall that $X$ is locally acyclic away from $A$ and hence there exists some $L>0$ such that $X-N_L(A)$ is locally acyclic.
   Recall that $k(\gs)$ is the infimum of the diameter of chains bounding $M(\d \gs)$. 
   By induction on the dimension of $\gs$, we observe that  $k(\gs)$ goes to zero as $\diam(\gs)$ goes to $0$, given that $d(\gs, A)\geq L$ because of the local acyclicity of $X-N_L(A)$.
   By construction, $M(\gs)$ is of diameter at most $2k(\gs)$.
   Therefore $\diam(|M(\gs)|)$ goes to zero as $\diam(\gs)$ goes to $0$, given that $d(\gs, A)\geq L$.  This gives us~\eqref{i:filling 2}.
   \end{proof}

% \textcolor{red}{more detailed proof}
 \begin{Lemma}\label{l:close maps are homotopic}
    Assume that $A\subset X$ and $G\acts (X,A)$. 
    Let $f:\cf(X)\rightarrow \cf(X)$ be a $G$-equivariant chain map and $\rho_n:[0,\infty)\rightarrow [0,\infty)$ is some non-decreasing sequence of non-decreasing functions such that 
    \begin{itemize}
       \item $|f(\gs^n)|\subset N_{\rho_n(\diam(\gs^n))}(\gs^n)$ for any $n$-simplex $\gs^n$.
        \item There exists an $L\geq 0$ such that for every $x\in X-N_L(A)$, and a neighborhood $U$ of $x$, there is a neighborhood $W\subset U$ of $x$ such that $\Supp{f(\gs^n)}\subset U^{n+1}$ whenever $\gs^n\in {W}^{n+1}$.
    \end{itemize}
     Then $f$ and the identity map on $\cf(X)$  are chain homotopic via a $G$-equivariant chain homotopy $H_n:\cf[n](X)\rightarrow \cf[n+1](X)$ with an associated non-decreasing sequence of non-decreasing functions $\rho'_n:[0,\infty)\rightarrow [0,\infty)$ such that
     \begin{enumerate}
         \item \label{support D} $|H_n(\gs)|\subset N_{\rho'_n(\diam(\gs))}(\gs)$ for any $n$-simplex $\gs^n$.         
         \item \label{D local} There exists an $L'\geq 0$ such that for every $x\in X-N_{L'}(A)$, and a neighborhood $U$ of $x$, there is a neighborhood $W\subset U$ of $x$ such that $\Supp{H_n(\gs^n)}\subset U^{n+2}$ for all $\gs^n\in {W}^{n+1}$.
         \end{enumerate}
     \end{Lemma}
% \textcolor{red}{AB:Expanded the statement of the lemma to make it more understandable}
 \begin{proof}
     We can define $H_n$ by induction on the dimension $n$.
         Define $H_0(x):=(x,f(x))$. 
         Note that $H_0$ is $G$-equivariant and  $f(x)-x=\d H_0(x)$. 
Since $f(x)\subset N_{\rho_0(\diam({x}))}(x)$, we get $H_0(x)\subset N_{\rho_0(\diam \{x\})}(x)$. We define $\rho_0':=\rho_0$.

     Suppose we have already defined $H_m:\cf[m](X)\rightarrow \cf[m+1](X)$ and the non-decreasing map $\rho_m'$ such that $H_m$ is $G$-equivariant, $H_m(\gs)\subset N_{\rho'_m(\diam(\gs))}(\gs)$ and $\d H_m(\gs)+D_{m-1}\d(\gs)=i(\gs)-f(\gs)$ for any $m\leq n$.
     To define $H_{n+1}(\gs)$ for an $(n+1)$-simplex $\gs$, consider the following chain 
     \[c_\gs:=\gs-f(\gs)-H_n\d(\gs).
     \]
     By induction hypothesis on $H_n$, $c$ is a cycle.
     Take a vertex $b$ from the chain $c$, and consider the cone operator 
     \[T_b:\cf[n+1](X)\rightarrow \cf[n+2](X), \quad (x_0,\ldots,x_{n+1})\mapsto (b,x_0,\ldots,x_{n+1}).
     \]
     Define $H_{n+1}(\gs):=T_b(c_\gs)$.
     It follows that  
     \begin{align*}
         \d H_{n+1}(\gs)=c_\gs-T_b(\partial c_\gs)&=c_\gs
         =\gs-f(\gs)-H_n\d(\gs).
    \end{align*}
    To make $H_{n+1}$ equivariant, we first define it on elements from each orbit class of $(n+1)$-simplices and then extend it $G$-equivariantly.
    
We now focus on the support of $H_{n+1}(\gs)$.
Suppose $\diam(\gs)=r$.
     By induction hypothesis, $H_n(\tau)\subset N_{\rho'_n}(\tau)$ for any $n$-simplex $\tau$.
     Since $\rho'_n$ is a non-decreasing function, $\rho'_n(r)\geq \rho'_n(\diam(\tau))$ for any $\tau\in |\d \gs|$.
    We have
     \begin{align*}
         |H_n(\d \gs)|\subset \cup_{\tau\in |\d \gs|} N_{\rho'_n(\diam(\tau))}(\tau)
         &\subset N_{\rho'_n(r)}(|\partial\gs|)\\
         &\subset N_{[\rho'_n(r)+r]}(\gs) \quad \text{[by Lemma~\ref{l:d(simplex,boundary)}]}.
         \end{align*}
Also, $\gs$ and $|f(\gs)|$ are subsets of $N_{\rho_{n+1}(r)}(\gs)$.
If we define,
         \[
         q(x):=\max\{\rho_n'(x)+x,\rho_{n+1}(x)\}
         \]
    then $|c_\gs|\subset N_{q(r)}(\gs)$ and consequently 
     \begin{align*}
         \diam(|c_\gs|)\leq \diam(\gs)+2q(r)\leq r+2L(r)
    \end{align*}
     Since $|T_b(c_\gs)|\subset N_{\diam(|c_\gs|)}(|c_\gs|)$, we obtain
     \begin{align*}
         |H_{n+1}(\gs)|=|T_b(c_\gs)|&\subset N_{\diam(|c_\gs|)}(|c_\gs|)\\
         &\subset N_{\diam(|c_\gs|)}[N_{q(r)}(\gs)]\\
         &\subset N_{[\diam(|c_\gs|)+q(r)]}(\gs)\subset N_{[r+3q(r)]}(\gs)   
    \end{align*}
  Letting $\rho'_{n+1}(r):=r+3q(r)$, we get property~\eqref{support D}.   
  
  To get property~\eqref{D local}, note that vertex of any  simplex in $|H_n(\gs)|$ is either a vertex of $\gs$ or a vertex of some simplices in $|f(\tau)|$ where $\tau$ is some sub-simplex of $\gs$. The claim then follows from the analogous property of the map $f$.
\end{proof}

Let $\cU$ denotes an open cover $X$.
We say $\cU$ is $G$-invariant, if for any $U\in \cU$, $gU\in \cU$ for all $g\in G$.
Let $\csu(X)$ be the chain complex generated by singular simplices supported in some open set in $\cU$. 
  Let $V:\cs(X) \to \cf(X)$ be the forgetful map, which maps a singular simplex to its vertices.
 
 To prove Theorem~\ref{t:coa=end 1}, we will need to fill in simplices by singular chains in $\csu(X)$ for some  $G$-invariant cover $\cU$.
More precisely, we need the following.

\begin{Prop}\label{l:cs} 
Suppose $X$ is uniformly acyclic away from $A$ and locally acyclic away from $A$ and $G\acts (X,A)$. 
Let $\cU$ be a $G$-invariant open cover of $X$.
Then there exist two non-decreasing sequences of functions $\mu_n,\rho_n:[0,\infty)\rightarrow [0,\infty)$,
 a $G$-equivariant chain map $S: \cff(X) \to \csu(X)$ where 
 \[
\cff[n](X)=\langle \gs^{n} \mid d(\gs^{n},A) \geq  \mu_{n}(\diam \gs^{n}) \rangle \subset \cf[n](X)
\]
 and a $G$-equivariant chain homotopy $H: \cff(X) \to \cf[*+1](X)$  between $VS:\cff(X)\rightarrow \cf(X)$ and the inclusion map so that
 
\begin{enumerate}

 \item \label{i:S} $\Supp{VS(\gs^{n})} \subset N_{\rho_n(\diam(\gs^n))}( \gs^{n})$ for any $n$-simplex $\sigma^n$
 \item \label{i:H} $\Supp{H(\gs^{n})}  \subset N_{\rho_n(\diam(\gs^n))}( \gs^{n})$ for any $n$-simplex $\gs^n$. 
 \item \label{i:local} There exists $r>0$ so that for every $x\notin N_r(A)$, there is a neighborhood $W$ of $x$ such that
 for all $\gs^n\in W^{n+1}$, $H(\gs^n)\in\csu[n+1](X)$. 
 \end{enumerate}
\end{Prop} 
\begin{proof}
Because of the assumptions on $X$, we can invoke Lemma~\ref{l:filling}, which outputs a $G$-equivariant map $M:\cff(X)\rightarrow \cs(X)$ satisfying~\eqref{i:filling 1} and~\eqref{i:filling 2} from Lemma~\ref{l:filling}.

Next, we choose a barycentric subdivision map $P:\cs(X)\rightarrow \csu(X)$.
We can choose $P$ in a $G$-equivariant way by the same trick as before: in each dimension, first define it on an element from each orbit class of simplices  and then extend $G$-equivariantly. 
Note that $P$ satisfies the following 
\begin{equation*}\label{e:P}
|P(\gs)|\subset N_{\diam(\gs)}(|\gs|) \quad \text{for any $\gs$} \tag{*}
\end{equation*}
We define $S:=P M$. Since both $P$ and $M$ are $G$-equivariant, so is $S$.

 Since $M$ satisfies Lemma~\ref{l:filling}\eqref{i:filling 1} and $P$ satisfies~\eqref{e:P},
we obtain that 
\[|VS(\gs^n)|\subset N_{\rho_n(\diam(\gs^n))}( \gs^{n})\] for some function non-decreasing sequence of function $\rho'_n:[0,\infty)\rightarrow [0,\infty)$.

Moreover, since $M$ satisfies~\eqref{i:filling 2} from Lemma~\ref{l:filling} and $P$ satisfies~\eqref{e:P}, there exists an $L\geq 0$ such that for every $x\in X-N_L(A)$, and a neighborhood $U$ of $x$, there is a neighborhood $W\subset U$ of $x$ such that $\Supp{VS(\gs^n)}\subset U^{n+1}$ whenever $\gs^n\in {W}^{n+1}$.

In conclusion,
$VS$ satisfies the hypothesis of Lemma~\ref{l:close maps are homotopic}.
Applying Lemma~\ref{l:close maps are homotopic} to the map $VS$ we obtain a chain homotopy $H_*$ between $VS$ and the $id$ that satisfies property~\eqref{i:H} and property~\eqref{i:local}.

% Finally to prove~\eqref{i:local}, note that the map $M$ satisfies the property~\eqref{i:local} by Lemma~\ref{l:filling}\eqref{i:filling 2}, and hence so does $S=PM$. Consequently $VS$ satifies property~\eqref{i:local}.
% By construction, $|H(\gs)|$ has same the set of vertices as $|VS(\gs)|$ for any $\gs$.
% Therefore $H$ satisfies~\eqref{i:local}.
\end{proof}

\begin{Remark}\label{r:filling}
    If $X$ is uniformly acyclic, then  we can take $\cff(X)$ to be $\cf(X)$ in Proposition~\ref{l:cs}. 
    Also, note that the local acyclicity away from $A$ was only used to ensure property~\eqref{i:local}.
    So, if the space is just uniformly acyclic away from $A$, we still get $H$ and $S$ satisfying  the properties~\eqref{i:S} and~\eqref{i:H}.
\end{Remark}
  
\begin{proof}[Proof of Theorem~\ref{t:coa=end 1}]
	Consider the following long exact sequence 
 \[
 \cdots\rightarrow \eHb[*-1](X-A)\rightarrow \H(\eCb(X-A)/\eCx(X-A))\rightarrow \eHx(X-A)\rightarrow \eHb(X-A)\rightarrow
 \]
  Therefore it is enough to show that $\H(\eCb(X-A)/\eCx(X-A))=0$.
   In other words, for $\phi \in \eCb[n](X-A)$ with $d\phi \in \eCx[n+1](X-A)$ we need to find $\psi \in \eCb[n-1](X-A)$ so that $\phi - d\psi \in \eCx[n](X-A)$.
  
  Our goal is to apply Proposition~\ref{l:cs}. In order to do that, we first need to choose a $G$-invariant cover of $X$. Let $r>0$ and let $U$ be union of all the balls of radius $r$ that are centered at points in $A\cap ||\phi||$.
	Since both $A$ and $||\phi||$ are $G$-invariant, $U$ is $G$-invariant: $g U=U$ for all $g\in G$.
	For each $x \in X - ||\phi || $, choose a metric neighborhood $U_{x}$ with diameter $\leq 1$ such that $U_{x}^{n+1} \cap \Supp{\phi}=\emptyset$.
	Since $\phi$ is $G$-equivariant, we can choose the association $x\mapsto U_x$ so that $U_{gx}=g U_x$ for all $g\in G$.
	Let $\cU$ denote the collection of $U_{x}$ together with $U$.
	By construction, this is a $G$-invariant cover.

Now we apply Proposition~\ref{l:cs} to this setup  which outputs a complex $\cff(X)$, a $G$-equivariant chain map $S:\cff(X)\rightarrow \cf^{\cU}(X)$ and a $G$-equivariant chain homotopy $H:\cff(X)\rightarrow \cf[*+1](X)$ between $VS$ and the inclusion map that satisfy property~\ref{i:S}, \ref{i:H}, and \ref{i:local} from Proposition~\ref{l:cs}.

We define a linear map $D:  \cf(X) \to \cf[*+1](X)$ by setting 
	\[
	D(\gs^{n})=
		\begin{cases}
 			H(\gs^n) &\text{ if } \gs^n\in \cff[n](X)\\
			0 &\text{ otherwise.}
		\end{cases}	
	\]
	
We define $\tau:\cf(X)\rightarrow \cf[*+1](X)$ as follows
\[\tau=id+\partial D+D \partial. 
\]
If $\gs\in \cff(X)$, then by construction $D(\gs)=H(\gs)$ and therefore $\tau(\gs)=VS(\gs)$.
Suppose $\tau^*$ denote the dual of $\tau$.
By applying $\tau^*$ on $\phi$, we get the following
	\[
		\tau^{*}\phi = \phi+ d D^{*}\phi+ D^{*} d \phi.
	\]
	
	We claim that $\tau^{*}\phi\in \eCx[n](X-A)$.
	If $\diam(\gs^n)\leq k$ and $d(\gs^n,A)>\mu_n(k)$, then $\gs^n \in \cff[n](X)$ and hence $\tau(\gs^n)=VS(\gs^n)$.
	Moreover, if $\gs^n$ is outside of the $\rho_{n}(k)$-neighborhood of $U^{n+1}$, then $|\tau(\sigma^n)|$ does not touch $U^{n+1}$ because $|\tau(\gs^n)|=|VS\gs^n|\subset N_{\rho_n(k)}(\gs^n)$ by property~\ref{i:S} from Proposotion~\ref{l:cs}. 
 This implies
	 $(\tau^{*}\phi)(\gs^n)=\phi(\tau (\gs^n))=0$.
  In other words, $\gs^n\notin |\tau^*\phi|$.
  It follows that, 
  \[|\tau^*\phi|\cap N_k(\gD)\subset N_{\mu_n(k)}(\gD_A)\cup N_{\rho_n(k)}(U^{n+1}).\]
  Since $U\csubset A$, we have 
  \[
  |\tau^*\phi|\cap N_k(\gD)\csubset \gD_A.
  \]
  This proves the claim.

 Next, we claim that  $D^*(\phi)\in \eCb[n-1](X-A)$. Since $H$ satifies property~\eqref{i:local} from proposition~\ref{l:cs}, we can choose a set $N_r(A)$ containing $U$ such that for any $x\notin N_r(A)$, there is  a neighborhood $W_x$ of $x$ so that $|H(\gs^n)|\subset U_x^{n+2}$ for any $\gs^n\in W_x^{n+1}$.
Hence for any $x\notin  N_r(A)\cup ||\phi||$, we have $|H(\sigma^n)|\notin |\phi|$ and hence $|D(\gs^n)|\notin |\phi|$ for all $\sigma^n \in W_x^{n+1}$.
Therefore, $||D^*(\phi)||\subset ||\phi||\cup N_r(A)$. The claim follows since $\phi\in \Cb(X)$.
	
	Finally, we claim that $D^{*}d(\phi)\in \eCx[n](X-A)$.
By construction of $D$, if $d(\gs^n,A)\geq\mu_n(k)$ then $\Supp{D\gs^n}\subset N_{\rho_n(k)}(\gs^n)$ where $k$ is the diameter of $\gs^n$.
	Since $d(\phi)\in \eCx(X-A)$, the claim follows.

	Now setting $\psi=-D^{*}(\phi)$ we get what we want.
\end{proof}

 We can now prove Theorem~\ref{c:equicoa=coaquo}.

 \begin{proof}[Proof of Theorem~\ref{c:equicoa=coaquo}]
We observe that $\eCx(X-A;R)=\mathscr{C}_G^*(X;R)$ when $X$ is coarsely contained in $A$.
Hence Theorem~\ref{c:equicoa=coaquo}\eqref{i:equicoa=coaquo1} follows from Lemma~\ref{l:basic equiv complex}.
  
  Theorem~\ref{c:equicoa=coaquo}\eqref{i:equicoa=coaquo2} follows immediately from Proposition~\ref{p:bdd=singular2} and  Theorem~\ref{t:coa=end 1}.

 \end{proof}

\section{Coarse cohomology of the configuration space}\label{s:coarse configuration}

Let $(X,d)$ be a metric space. Equip $X^2=X\times X$ with the sup metric.
 Consider the $\zz_2$-action on $X^2$ that flips the coordinates.
 Since the fixed point set for this action is the diagonal subspace $\delta(X)=\{(x,x)\mid x\in X\}\subset X^2$,
  we have $\zz_2\acts (X^2,\delta(X))$.
 Let $R$ be an abelian group with a $\zz_2$-action.
 We define \emph{coarse cohomology of the two-points configuration space of $X$} with coefficients in $R$ to be the cohomology of the complex $\Cx_{\zz_2}(X^2 -\delta(X);R)$.
 From now on, for the sake of simplicity, we will omit the term "two-point" from our terminology and refer to it simply as the coarse cohomology of the configuration space.
 
 Theorem \ref{c:equicoa=coaquo} immediately gives us the following.
\begin{Prop}\label{compute HX(conf)}
 If $X$ is unbounded, uniformly acyclic, and  locally acyclic with coefficients in $R$ where $R$ is an abelian group with trivial $\zz_2$ action. Then
 \begin{align*}
     \Hx_{\zz_2}(X^2 -\delta(X);R)=\begin{cases}
     0 & \text{if $*=0$},\\
     \varinjlim\rH[*-1]((X^2-N_r(\delta(X)))/\zz_2;R) & \text{otherwise}.
     \end{cases}
 \end{align*}.
 \end{Prop}

  \begin{Example}\label{cconfR^n}
 $\rr^n$ satisfies the hypothesis of   Proposition~\ref{compute HX(conf)}. 
Moreover, for any $r$, there is a $\zz_2$-equivariant deformation retraction of $(\rr^{n})^2-\delta(\rr^n)$ to $(\rr^{n})^2-N_r(\delta(\rr^n))$. 
Therefore, applying Proposition~\ref{compute HX(conf)} we obtain the following where the coefficients group is  $\zz_2$ with the trivial $\zz_2$-action.
\begin{align*}
  \Hx_{\zz_2}((\rr^n)^2-\delta(\rr^n);\zz_2)&=\begin{cases}
  0 & \text{if $*=0$}\\
  \rH[*-1](((\rr^{n})^2-\delta(\rr^n))/\zz_2;\zz_2) & \text{otherwise}
  \end{cases}\\
  &=\begin{cases}
  0 & \text{if $*=0$}\\
  \rH[*-1](\rr P^{n-1};\zz_2) & \text{otherwise}
  \end{cases}\\
  &=\begin{cases}
   \zz_2 & \text{if $2\leq *\leq n$}\\ 
  0 & \text{otherwise}
  \end{cases}
  \end{align*}
  \end{Example}
  
 \begin{Example}\label{conflinfty}
 Recall that $\ell^\infty$ 
 is the space of bounded sequences of real numbers with the sup-norm metric.
 The space $(\ell^\infty)^2-\delta(\ell^\infty)$, $\zz_2$-equivariantly deformation retracts  to $(\ell^\infty)^2-N_r(\delta(\ell^\infty))$.
 Since $(\ell^\infty)^2-\delta(\ell^\infty)$ is acyclic, $((\ell^\infty)^2-\delta(\ell^\infty))/\zz_2$ is a classifying space for $\zz_2$ and hence is
 homotopy equivalent to $\rr P^\infty$. Hence arguing as in Example~\ref{cconfR^n} we obtain the following.
 \[
  \Hx_{\zz_2}((\ell^\infty)^2-\delta(\ell^\infty);\zz_2)=\begin{cases}
   \zz_2 & \text{for} \quad *\geq 2\\ 
  0 & \text{otherwise}
  \end{cases}
  \]
 \end{Example}

Next we describe the maps between two metric spaces that induce map between the corresponding coarse cohomology of the configuration spaces. In the topological setting, an injective continuous map induces a  map between the cohomology of the corresponding configuration spaces. In the coarse setting, the role of injective continuous maps are played by \emph{coarse expanding} maps which we define next.

\begin{Definition}
 A map $f:(X,A)\rightarrow (Y,B)$ between pairs is called relatively proper  if $f^{-1}(N_r(B))\csubset A$ for any $r$.
 
 A map $f:(X,A)\rightarrow (Y,B)$ between pairs is called relatively coarse if $f$ is relatively proper and there exists a non-decreasing function $\mu:[0,\infty)\rightarrow [0,\infty)$ such that $d(f(x),f(y))\leq \mu(d(x,y))$ for all $x,y\in X$.

 A map $f:X\rightarrow Y$ is a coarse expanding map if the induced map $(x,y)\mapsto (f(x),f(y))$ from $(X^2,\delta(X))$ to $(Y^2,\delta(Y))$ is a relatively coarse map.  \end{Definition}

\begin{Example}
Recall that a map $f:X\rightarrow Y$ between two metric spaces is said to be a \emph{coarse embedding} if there exist two proper non decreasing maps 
 $\rho_-,\rho_+:[0,\infty)\rightarrow [0,\infty)$ such that 
 \[\rho_-(d(x,y))\leq d(f(x),f(y)) \leq \rho_+(d(x,y)) \quad \text{for all} \quad x,y\in X.
 \]
 One can see that any coarse embedding map  is a coarse expanding map.

\begin{Example}
    An example of a map that is not coarse expanding is the map $x\mapsto |x|$ between real numbers. The reason is that the map $(x,y)\mapsto (|x|,|y|)$ from $(\rr^2,\delta(\rr))$ to itself is not relatively proper. 
\end{Example}
 
\end{Example}
  We now recall the following from~\cite{BB20}.
 \begin{Lemma}\label{p:relcoarsemap}
    For any abelian group $R$, a relatively coarse map $f:(X,A)\to (Y,B)$ between pairs induces chain map $f^*:\Cx(Y-B;R)\to \Cx(X-A;R)$ by the canonical formula
     \[
(f^*\phi)(x_0,x_1,\ldots,x_n)=\phi(f(x_0),f(x_1),\ldots,f(x_n)).
     \]
 \end{Lemma}

 The following lemma follows immediately from Lemma~\ref{p:relcoarsemap}.
 \begin{Lemma}\label{exp map}
      If $f:X\rightarrow Y$ is a coarse expanding map and $R$ is an abelian group, then the  map $(x,y)\mapsto (f(x),f(y))$ induces a map $f^*:\Hx_{\zz_2}(Y^2-\delta(Y);R)\rightarrow \Hx_{\zz_2}(X^2 -\delta(X);R)$.
 \end{Lemma}
  
 \section{Coarse van Kampen obstruction}\label{s:cvk}

 In this section, we find an obstruction to the existence of coarse expanding maps between two metric spaces.
Our key observation is the next proposition.
 
 \begin{Prop}\label{T:universality}
 Any two coarse expanding maps from $X$ to $\ell^\infty$ induce the same map
  from $\Hx_{\zz_2}((\ell^\infty)^2-\delta(\ell^\infty);R)$ to  $\Hx_{\zz_2}(X^2 -\delta(X);R)$.
 \end{Prop}
 
 \begin{proof} 
For convenience, we will suppress the coefficient $R$ from the notation in the proof.

Suppose $f,g:X^2 \rightarrow (\ell^\infty)^2$ are two maps induced by two coarse expanding maps from $X$ to $\ell^\infty$.
For convenience,  we will denote the induced map between $\cf(X^2)$ and $\cf((\ell^\infty)^2)$ by $f$ and $g$ as well.
 Let $f^*$ and $g^*$ be the corresponding map from $\Cx_{\zz_2}((\ell^\infty)^2-\delta(\ell^\infty))$ to  $\Cx_{\zz_2}(X^2-\delta(X))$.
 We will show that there is a chain homotopy $\Cx_{\zz_2}((\ell^\infty)^2-\delta(\ell^\infty))\rightarrow  \Cx[*+1]_{\zz_2}(X^2-\delta(X))$ between $f^*$ and $g^*$.
 
 Since $(\ell^\infty)^2$ is uniformly contractible, the proof of  Proposition~\ref{l:cs} (see remark~\ref{r:filling}) gives us a $\zz_2$-equivariant  chain map $S:\cf((\ell^\infty)^2)\rightarrow \cs((\ell^\infty)^2)$ such that $|S(\gs^n)|\subset N_{\rho_n(\diam(\gs^n))}(\gs^n)$ for some non-decreasing sequence of functions $\rho_n:[0,\infty)\rightarrow [0,\infty)$.
 Moreover, by Proposition~\ref{l:cs}, the composition $V S$ is chain
 homotopic to the identity map by 
 a $\zz_2$-equivariant chain homotopy $H$ such that $|H(\gs^n)|\subset N_{\rho_n(\diam(\gs^n))}(\gs^n)$.

 We now construct a chain homotopy $D_*:\cf(X^2)\rightarrow \cs[*+1]((\ell^\infty)^2)$ between the two maps $S\circ f, S\circ g:\cf(X^2)\rightarrow \cs((\ell^\infty)^2)$ with certain properties: in particular we want the homotopy to avoid $\delta(\ell^\infty)$.

For convenience let us fix some notations.
 For a singular chain $c\in \cs(X)$, let \[\size(c):=\sup_{\tau \in |c|}\{\diam(\tau)\}.
 \]
 For a simplex $\gs\in (X^2)^*$,
 let 
 \[r_{\sigma}:=\min\{d(|S(f(\gs))|,\delta(\ell^\infty)),d(|S(g(\gs))|, \delta(\ell^\infty))\}
 \]
 and 
 \[R_{\sigma}:= \max\{d(|S(f(\gs))|,\delta(\ell^\infty)),d(|S(g(\gs))|,\delta(\ell^\infty))\}.
 \]
 Let $C_\gs\subset (\ell^\infty)^2$ be the following annulus around $\delta (\ell^\infty)$:
 \[\{z\mid r_\gs\leq d(z,\delta(\ell^\infty))\leq R_\gs\}.
 \]
We note that $C_\gs$ is acyclic because it is homotopy equivalent to $S^\infty$.

We claim that there exists a $D_i:\cf[i](X^2)\rightarrow \cs[i+1]((\ell^\infty)^2)$ such that for any $i$-simplex $\gs$ the following hold.
\begin{enumerate}
\item \label{prop1} $S\circ f(\gs)-S\circ g(\gs)=\partial D_i(\gs)+D_{i-1}\partial(\gs)$
   \item \label{prop2} $\Supp{D_i(\sigma)} \subset C^{i+2}_\gs$. 
 \item \label{prop3} $\size(D_i(\sigma)) \leq \size(S(f(\sigma))-S(g(\sigma))-D_{i-1}\partial(\sigma))$
 \item \label{prop4} $D_i$ is $\zz_2$-equivariant.
\end{enumerate}

We will defer the construction of $D_i$ to Lemma~\ref{D}. Assuming the existence of such $D_i$, we now complete the proof of the proposition.

Let $V:\cs\rightarrow \cf$ be the map that sends a singular simplex to its vertices.
 Since $V\partial=\partial V$,
applying $V$ on both sides of \eqref{prop1}, we have the following
\begin{equation}\label{h1}
    VS(f-g)=\partial VD_i +VD_{i-1} \partial \hfill \tag{*}
\end{equation}
Recall that $V S$ is homotopic to the identity map with a $\zz_2$-equivariant chain homotopy $H$ such that $|H(\gs^n)|\subset N_{\rho_n(\diam(\gs^n))}(\gs^n)$. So, we have
\begin{equation}\label{h2}
VS=id+\d H_i+H_{i-1}\d \hfill \tag{**}
\end{equation}
Combining \eqref{h1} and \eqref{h2}, we obtain
\[
f-g=\d[VD_i+H_i(f-g)]+[VD_{i-1}+H_{i-1}(f-g)]\d.
\]
Dualizing the above we get $f^*$ and $g^*$ are chain homotopic via the cochain homotopy $(VD)^*+(f-g)^*H^*$.
To complete the proof, we need to show that this cochain homotopy maps $\Cx_{\zz_2}((\ell^\infty)^2-\delta(\ell^\infty))$ to $\Cx[*-1]_{\zz_2}((\ell^\infty)^2-\delta(\ell^\infty))$

Since $|H(\gs^n)|\subset N_{\rho_n(\diam(\gs^n))}(\gs^n)$ and $H$ is $\zz_2$-equivariant, the dual $H^*$ maps $\Cx_{\zz_2}((\ell^\infty)^2-\delta(\ell^\infty))$ to $\Cx[*-1]_{\zz_2}((\ell^\infty)^2-\delta(\ell^\infty))$.
Consequently, $(f-g)^*H^*$ maps $\Cx_{\zz_2}((\ell^\infty)^2-\delta(\ell^\infty))$ to $\Cx[*-1]_{\zz_2}(X^2-\delta(X))$.
It is therefore enough to prove that $(VD)^*$ maps $\Cx_{\zz_2}((\ell^\infty)^2-\delta(\ell^\infty))$ to $\Cx[*-1]_{\zz_2}(X^2-\delta(X))$.
Let $\phi\in \Cx_{\zz_2}((\ell^\infty)^2-\delta(\ell^\infty))$ and $\gs\in |(VD)^*(\phi)|\cap N_r(\gD)$ for some $r\geq 0$.
 Since $\gs\in N_r(\gD)$, and $f$ and $g$ are coarse expanding maps, we have $|S(f(\gs))|\subset N_s(\gD)$ and $|S(g(\gs))|\subset N_s(\gD)$ for some $s$ that depends only on $r$.
 Property \eqref{prop3} then implies that $|VD(\gs)|\subset N_s(\gD)$ for some $s$ that depends only on $r$.
Since $\phi\in \Cx_{\zz_2}((\ell^\infty)^2-\delta(\ell^\infty))$ and $\phi(VD(\gs))\neq 0$, it then follows that $|VD(\gs))|\subset N_t(\gD_{\delta(\ell^\infty)})$ where $t$ depends only on $s$ and hence depends only on $r$.
It now follows from property \eqref{prop2} that $\gs\in N_p(\Delta_{\delta(X))})$ for some $p$ that depends only on $t$ and hence only on $r$.
Hence, we proved that for each $r\geq 0$, there exists $p\geq 0$ such that $|(VD)^*(\phi)|\cap N_r(\gD)\subset N_p(\Delta_{\delta(X)})$. 
Finally, $(VD)^*(\phi)$ is $\zz_2$-equivariant by property~\eqref{prop4}.
Hence, $(VD)^*(\phi)\in \Cx[*+1]_{\zz_2}(X^2-\delta(X))$. This finishes the proof.

\end{proof}

 \begin{Lemma}\label{D}
    Let $f$ and $g$ are two coarse expanding maps from $X$ to $\ell^\infty$. 
    Let $S:\cf((\ell^\infty)^2)\rightarrow \cs((\ell^\infty)^2)$ be $\zz_2$-equivariant chain map such that $|S(\gs^n)|\subset N_{\rho_n(\diam(\gs^n))}(\gs^n)$ for some sequence of functions $\rho_n:[0,\infty)\rightarrow [0,\infty)$.    
    Then, for each $i$, there exists  $D_i:\cf[i](X^2)\rightarrow \cs[i+1]((\ell^\infty)^2)$ such that for any $i$-simplex $\gs$ the following hold.
\begin{enumerate}
\item \label{lprop1} $S\circ f(\gs)-S\circ g(\gs)=\partial D_i(\gs)+D_{i-1}\partial(\gs)$
   \item \label{lprop2} $\Supp{D_i(\sigma)} \subset C^{i+2}_\gs$. 
 \item \label{lprop3} $\size(D_i(\sigma)) \leq \size(S(f(\sigma))-S(g(\sigma))-D_{i-1}\partial(\sigma))$
 \item \label{lprop4} $D_i$ is $\zz_2$-equivariant.
\end{enumerate}
\end{Lemma}
\begin{proof}
    We first define $D_0:\cf[0](X^2)\rightarrow \cs[1]((\ell^\infty)^2)$.
  For a $0$-simplex $\gs\in X^2$, join $f(\gs)$ and $g(\gs)$ by a path $\gamma$ in $C_\gs$.
Such a path exists because the annulus $C_\gs$ is contractible. 
Next, we subdivide $\gamma$ so that each subarc is of diameter $\leq 1$. 
We define $D_0(\gs)$ to be this subdivided path.
By construction, $D_0$ satisfies the first three desired properties.
To make it $G$-equivariant, we first define $D_0$ on a simplex from each $G$-orbit of simplices and then move it $\zz_2$-equivariantly.
It is straightforward to see that $D_0$ still satisfies the first property because $\d$ commutes with the $G$-action.
It satisfies the second property because $gC_\gs=C_{g\gs}$ for any $g\in \zz_2$.
$D_0$ satisfies the third property because isometric action preserves size of chains.

Inductively, let us  assume that $D_i:\cf[i](X^2)\rightarrow \cs[i+1]((\ell^\infty)^2)$ is already  defined for all $i\leq n$ with the desired properties.  
To define $D_{n+1} (\sigma)$, let $K=S(f(\sigma))-S(g(\sigma))-D_n(\partial (\sigma))$. By induction hypothesis \eqref{lprop1}, $\partial K= D_{n-1}(\partial^2( \sigma))=0$ and hence $K$ is a cycle. 
By \eqref{lprop2},  $\Supp{K} \subset C_\gs$.
Since  $C_\gs$ is acyclic,
 there exists a singular chain $c$ supported in $C_\gs$ such that $\partial c = K$. After applying the appropriate subdivision, we can make $c$ to satisfy $\size(c) \leq \size(K)$ without changing its boundary. 
Define $D_{n+1} (\sigma)$ to be that $c$.
By construction, $D_{n+1}$ satisfies conditions \eqref{lprop1}, \eqref{lprop2}, and \eqref{lprop3}. 
To make $D_{n+1}$ satisfy \eqref{lprop4}, we can use the same trick as before: first define it on a simplex from each $\zz_2$-orbit of $(n+1)$-simplices and then extend equivariantly.
For the similar reason as $D_0$, this $D_{n+1}$ has all the desired four properties.
\end{proof}

 \begin{Coarse van Kampen obstruction class}
 Let $X$ be a separable metric space and $g:X\rightarrow \ell^\infty$ be a coarse expanding map. Such a map exists because any separable metric space admits an isometric embedding into $\ell^\infty$ by the work of Fr\'{e}chet (\cite{Frechet}).
 We consider the induced map $g^*:\Hx[n]_{\zz_2}((\ell^\infty)^2-\delta(\ell^\infty);\zz_2)\rightarrow \Hx[n]_{\zz_2}(X^2-\delta(X);\zz_2)$.
 Recall from  Example~\ref{conflinfty} that $\Hx[n]_{\zz_2}((\ell^\infty)^2-\delta(\ell^\infty);\zz_2) = \zz_2$ if $n \geq 2$.
 Let $e^{n}$ be the nontrivial element in $\Hx[n]_{\zz_2}((\ell^\infty)^2-\delta(\ell^\infty))$ for $n\geq 2$.
 Proposition~\ref{T:universality} implies that $g^*(e^n)$ depends only on the space $X$, not on $g$.
We call the class  $g^*(e^{n})$ to be the \emph{$n^{th}$ degree coarse van Kampen obstruction class} of $X$ and denote it by $cvk^n(X)$ where $n\geq 2$.
\end{Coarse van Kampen obstruction class}

\begin{Assumption}
From now on all our metric spaces will be separable. In the non-equivariant setting, our coefficient group will be  $\zz_2$ and 
in the equivariant setting, the coefficient group will be $\zz_2$ with the trivial $\zz_2$-action unless stated otherwise.
In those cases, we will omit the coefficient from the notation.
\end{Assumption}

\begin{Definition}[Coarse obstruction dimension]\label{cobdim def}
The \textit{coarse obstruction dimension} of a  space $X$, denoted by $\cobdim(X)$, is $0$ if $X$ is bounded, is $1$ if $cvk^n(X)=0$ for all $n$, and otherwise, it is the largest $n$ such that $cvk^{n}(X)\neq 0$.
 \end{Definition}
 
 Now we prove the main theorem of this section.
 \begin{Theorem}\label{t:cobdim increases}
 If $X$ admits a coarse expanding map into $Y$, then  $\cobdim(X)\leq \cobdim(Y)$.
 \end{Theorem}
 
 \begin{proof}
 If $\cobdim(X)=0$, then there is nothing to prove.
 
 If $\cobdim(X)= 1$, then $X$ is unbounded by definition. This means $Y$ is also unbounded and hence $\cobdim(Y)\geq 1$ by definition.

 Suppose $\cobdim(X)=n\geq 2$.
 Let $g:Y\rightarrow \rr^\infty$ be a coarse expanding map.
 Consider the following composition. 
 \[ X\times X \xrightarrow{f} Y\times Y \xrightarrow{g} \ell^{\infty}\times \ell^{\infty}
 \]
 By Proposition \ref{exp map}, the above maps induce the  following maps between coarse cohomology of the configuration spaces
 \[\Hx[n]_{\zz_2}((\ell^\infty)^2-\delta(\ell^\infty))\xrightarrow{g^*}\Hx[n]_{\zz_2}(Y^2-\delta(Y))\xrightarrow{f^*}\Hx[n]_{\zz_2}(X^2-\delta(X))
 \]
 Let $e^n\in \Hx[n]_{\zz_2}((\ell^\infty)^2-\delta(\ell^\infty))$ be the generator.
 Then $cvk^n(Y)=g^*(e^n)$ and $cvk^n(X)=f^*g^*(e^n))=f^*(cvk^n(Y))$.
By assumption $cvk^n(X)\neq 0$ and hence $cvk^n(Y)\neq 0$.
So, we get $\cobdim(Y)\geq n$.
 \end{proof}
  Recall that $X$ and  $Y$ are said to be \emph{coarsely equivalent} if there exists a coarse embedding map $f:X\rightarrow Y$ such that $Y\csubset f(X)$.
 One can observe that two coarsely equivalent spaces coarsely embed into each other. 
 As a consequence, the above theorem immediately yields the following.
 \begin{Corollary}
     If $X$ and $Y$ are coarsely equivalent, then $\cobdim(X)=\cobdim(Y)$.
 \end{Corollary}

In Example~\ref{cconfR^n}, we saw that $\Hx_{\zz_2}((\rr^n)^2-\delta(\rr^n))=0$ for all $*>n$. 
 Hence $\cobdim(\rr^n)\leq n$. Using Theorem~\ref{t:cobdim increases}, we obtain

 \begin{Corollary}\label{c:obstruction rn}
If $\cobdim(X)\geq n$, then $X$ does not admit a coarse expanding map into $\rr^{n-1}$.
 \end{Corollary}
 
 \section{Relation to the classical van Kampen obstruction}\label{relation between van kampens}
Let us recall the classical van Kampen obstruction class. 
Let $X$ be any topological space. 
Any continuous  embedding $f:X\hookrightarrow \rr^\infty$ induces a map 
\[f:\H((\rr^\infty)^2-\delta(\rr^\infty)/{\zz_2})\rightarrow \H((X^2-\delta(X))/\zz_2).\]
This map depends only on $X$ because the quotient map \[q:(\rr^\infty)^2-\delta(\rr^\infty)\rightarrow ((\rr^\infty)^2-\delta(\rr^\infty))/{\zz_2}\]
is a universal $\zz_2$-bundle.
Let $\eta^*\in \H((\rr^\infty)^2-\delta(\rr^\infty)/{\zz_2})$ be the generator.
 The cohomology class $f(\eta^*)$ is called the van Kampen obstruction class in degree $*$ and will be denoted by $vk^*(X)$.
Note that the quotient map $(\ell^\infty)^2-\delta(\ell^\infty) \rightarrow ((\ell^\infty)^2-\delta(\ell^\infty))/{\zz_2}$ can also be considered as a universal $\zz_2$-bundle because $(\ell^\infty)^2-\delta(\ell^\infty)$ is contractible.
Hence we can use $\ell^\infty$ instead of $\rr^\infty$ to define $vk^*(X)$.
 We use this viewpoint in the next proposition.

\begin{Prop}\label{compute cvk} 
Let $X$ be a metric space. Suppose $i:\H[*]((X^2-\delta(X))/\zz_2)\rightarrow \varinjlim \H[*]((X^2-N_r(\delta(X)))/\zz_2)$ is the map induced by inclusions $X^2-N_r(\delta(X))\hookrightarrow X^2-\delta(X)$ for each $r>0$. If $i(vk^{n-1}(X))$ is nontrivial for some $n\geq 2$, then $cvk^n(X)$ is nontrivial. 
 \end{Prop}
 \begin{proof}
  Let $f:X\rightarrow \ell^\infty$ be an isometry.
  Fix $n\geq 2$.
  Suppose $i(vk^{n-1}(X))\neq 0$.
  We consider the following commutative diagram where all the horizontal maps are canonically induced by $f$.  
  We want to show that the top horizontal map is nontrivial. 
  \[\begin{tikzcd}
  &\Hx[n]_{\zz_2}((\ell^\infty)^2-\delta(\ell^\infty))\arrow{r}{f^*}\arrow{d}{\cong} &\Hx[n]_{\zz_2}(X^2-\delta(X))\arrow{d}\\
  &\varinjlim\H[n-1](((\ell^\infty)^2-N_r(\delta(\ell^\infty)))/\zz_2)\arrow{r}{f^*} &\varinjlim \H[n-1]((X^2-N_r(\delta(X)))/\zz_2)\\
  &\H[n-1](((\ell^\infty)^2-\delta(\ell^\infty))/\zz_2)\arrow{r}{f^*}\arrow{u} &\H[n-1]((X^2-\delta(X))/\zz_2)\arrow{u}{i}
  \end{tikzcd}
  \]
  The top left isomorphism map is due to the Proposition~\ref{compute HX(conf)}.

  The image of the generator of $\H[n-1](((\ell^\infty)^2-\delta(\ell^\infty))/\zz_2)$ under the  bottom horizontal map is  $vk^{n-1}(X)$ if $n\geq 2$.
 Since $i(vk^{n-1}(X))\neq 0$, the commutativity of the above diagram implies that the middle horizontal map is nontrivial. 
 Again using commutativity of the diagram, we conclude that the top horizontal map is nontrivial.
  \end{proof}

Suppose  $X=K\times [0,\infty)/K\times \{0\}$ is the open cone on a finite simplicial complex $K$.
         A metric $d$ on $X$ is called \emph{expanding} if for any two disjoint simplices $\gs,\tau\in K$ and $S\geq 0$, there exists $r\geq 0$ such that $d(\gs\times [r,\infty),\tau\times [r,\infty))\geq S$.
\begin{Prop}\label{vk to cvk}
    Let $X$ be an open cone on a finite simplicial complex $K$ and $X$ is equipped with an expanding metric. If $vk^{n-1}(X)\neq 0$ for some $n\geq 2$, then $cvk^n(X)\neq 0$.
\end{Prop}
\begin{proof}
     Since the metric on $X$ is expanding, any representative of  $c\in \h(\Conf(X))$ can be homotoped to a cycle that lives in $\H((X^2-N_r(\delta(X)))/\zz_2)$ for any $r$.
         Considering the dual, this implies that the restriction map $\rH[*]((X^2-\delta(X))/\zz_2)\rightarrow \varinjlim \rH[*]((X^2-N_r(\delta(X)))/\zz_2)$ sends $vk^{n-1}(X)$ to a nontrivial element for any $r\geq 0$.
         Hence, $i(vk^{n-1}(X))$ is nontrivial where $i$ is as in Proposition~\ref{compute cvk}.
         Hence Proposition~\ref{compute cvk} implies $cvk^{n+1}(X)\neq 0$.
\end{proof}
 \begin{Example}[Bestvina--Kapovich--Kleiner obstruction]\label{BKK obstruction}
         One of the results of Bestvina, Kapovich, and Kleiner in~\cite{BKK} can be stated as follows: if there is class $c\in \h[n](\Conf(X))$ such that $vk^{n}(X)(c)\neq 0$, then $X$ with a proper, expanding metric cannot be coarsely embedded inside $\rr^n$.
         One can see that Proposition~\ref{vk to cvk} combined with Corollary~\ref{c:obstruction rn} recovers Bestvina--Kapovich--Kleiner's result.
\end{Example}
\begin{Remark}
Note that that $vk^{n-1}(X)\neq 0$ does not necessarily imply $cvk^n(X)\neq 0$. For example, take $X$ to be a unit disk in $\rr^2$, then $vk^1(X)\neq 0$. However, since $X$ is bounded, Proposition~\ref{p:bdd=singular2} implies $\Hx[2]_{\zz_2}(X^2-\delta(X))=0$, and therefore $cvk^2(X)=0$.
\end{Remark}
 
 \section{A Coarse Gysin sequence}\label{s:coarse gysin}
 Recall that coarse van Kampen obstruction class lives in $\Hx_{\zz_2}(X^2-\delta(X))$.
 In this section, we relate $\Hx_{\zz_2}$ to $\Hx$ and apply that to compute $\Hx_{\zz_2}(X^2-\delta(X))$ for certain $X$.

Suppose $\zz_2$ is acting on some metric space $X$ by isometries and $A$ is the subset of $X$ that is fixed by the action.
We consider the following exact sequence
\[
0\rightarrow \Cx_{\zz_2}(X-A)\xrightarrow{i} \Cx(X-A) \xrightarrow{p} \Cx_{\zz_2}(X-A)\xrightarrow{r} \C(A)\rightarrow 0
\]
where $i$ is the inclusion map and $p(\phi):\gs\mapsto \phi(\gs)+\phi(g\gs)$ where $g$ is the generator of $\zz_2$ and $r$ is the restriction map.
The image of $p$ consists of those cochains in $\Cx_{\zz_2}(X-A)$ that send any simplex supported on $A$ to zero.
It is easy to see that collection of such cochains gives a subcomplex of $\Cx_{\zz_2}(X-A)$.
We denote this complex by $\Cx_{\zz_2}(X-A,A)$, and the corresponding cohomology by $\Hx_{\zz_2}(X-A,A)$.
Hence, the above four-term short exact sequence splits into the following two short exact sequences.
\begin{align*}
    &0\rightarrow \Cx_{\zz_2}(X-A)\xrightarrow{i} \Cx(X-A) \xrightarrow{p} \Cx_{\zz_2}(X-A,A)\rightarrow 0\\
    &0\rightarrow \Cx_{\zz_2}(X-A,A)\xrightarrow{i} \Cx_{\zz_2}(X-A) \xrightarrow{r} \C(A)\rightarrow 0    
    \end{align*}

 These two short exact sequences give us two long exact sequences in the cohomology which we record as our next lemma.
 \begin{Lemma}\label{singular Gysin}
Let $X$ be a metric space and $\zz_2\acts (X,A)$ such that $A$ is the fixed point set of the action. 
Then we have the following two long exact sequences
 \begin{align}
      \cdots  \rightarrow \Hx_{\zz_2}(X-A) \rightarrow \Hx(X-A) \rightarrow \Hx_{\zz_2}(X-A,A) \rightarrow \Hx[*+1]_{\zz_2}(X-A)\rightarrow \cdots \\
      \cdots  \rightarrow \Hx_{\zz_2}(X-A,A) \rightarrow \Hx_{\zz_2}(X-A) \rightarrow \H(\C(A)) \rightarrow \Hx[*+1]_{\zz_2}(X-A,A)\rightarrow \cdots \label{i:second les} \end{align}
 \end{Lemma}

 \begin{Lemma}\label{l:rel=norm}
     Suppose $A\subset X$ and $X$ is not coarsely contained in $A$.
     If $\zz_2\acts (X,A)$ where $A$ is the fixed point set of the action, then
     \[
     \Hx_{\zz_2}(X-A,A)=\begin{cases}
     0 &\text{if $*=0$}\\
      \Hx_{\zz_2}(X-A)\oplus \zz_2 &\text{if $*=1$}\\     
      \Hx_{\zz_2}(X-A) & \text{if $*\geq 2$}
    \end{cases}     
    \]
\end{Lemma}

 \begin{proof}
  Since $X$ is not coarsely contained in $A$, we have 
  \[\Hx[0]_{\zz_2}(X-A,A)=\Hx[0]_{\zz_2}(X-A)=0.
  \]
 
 By Lemma~\ref{l:C is acyclic}, the homology of the complex $\C(A)$ vanishes everywhere except in the degree $0$ where it is $\zz_2$.
    Combining this with  the second long exact sequence of the Lemma~\ref{singular Gysin}, we obtain 
     \begin{multline}\label{e:rel=normal}
          0\rightarrow \zz_2
          \rightarrow \Hx[1]_{\zz_2}(X-A,A)\rightarrow \Hx[1]_{\zz_2}(X-A)\rightarrow 0\rightarrow  \\
         \cdots 0\rightarrow \Hx_{\zz_2}(X-A,A)\rightarrow \Hx_{\zz_2}(X-A)\rightarrow 0\cdots 
          \end{multline}
     where $*\geq 2$.

     Since the coefficient group is $\zz_2$, the first five terms give us a  split short exact sequence.
     Hence $\Hx[1]_{\zz_2}(X-A,A)=\Hx[1]_{\zz_2}(X-A)\oplus \zz_2$.
 It also follows from~\eqref{e:rel=normal} that  $\Hx(X-A,A)=\Hx(X-A)$ for all $*\geq 2$.
     
 \end{proof} 

 Hence under the hypothesis of Lemma~\ref{l:rel=norm}, we can rewrite the first long exact sequence of Lemma~\ref{singular Gysin} as follows.
 
 \begin{Lemma}[Coarse Gysin sequence]\label{coarsegysin}
 Suppose $\zz_2\acts (X,A)$ where  $A\subset X$ is the fixed point set of the action.
 Moreover, assume that $A$ is unbounded and $X$ is not coarsely contained in $A$. Then there is a long exact sequence of the following form:
 \begin{multline*}
     0\rightarrow  \Hx[0]_{\zz_2}(X-A) \rightarrow \Hx[0](X-A)\rightarrow 0\\
     \rightarrow \Hx[1]_{\zz_2}(X-A)\rightarrow \Hx[1](X-A)\rightarrow  \Hx[1]_{\zz_2}(X-A)\oplus \zz_2 \rightarrow  \cdots\\
       \rightarrow \Hx_{\zz_2}(X-A) \rightarrow \Hx(X-A) \rightarrow \Hx_{\zz_2}(X-A) \rightarrow \cdots 
  \end{multline*}

 where $*\geq 2$.
 \end{Lemma}

 We are now ready to state our main result of this section.
 \begin{Prop}\label{l:coarsegysin}
     Suppose $X$ is unbounded and $\zz_2$ is acting on $X^2$ by permuting the coordinates. Then there is a long exact sequence of the following form:
 \begin{multline}\label{e:coarsegysin}
 0\rightarrow \Hx[0]_{\zz_2}(X^2-\delta(X))\rightarrow \Hx[0](X^2-\delta(X))\rightarrow 0 \\
     \rightarrow  \Hx[1]_{\zz_2}(X^2-\delta(X))\rightarrow \Hx[1](X^2-\delta(X))\rightarrow  \Hx[1]_{\zz_2}(X^2-\delta(X))\oplus \zz_2 \rightarrow \cdots  \\
      \rightarrow \Hx_{\zz_2}(X^2-\delta(X)) \rightarrow \Hx(X^2-\delta(X)) \rightarrow \Hx_{\zz_2}(X^2-\delta(X)) \rightarrow \cdots 
\end{multline}
 
 where $*\geq 2$.
 \end{Prop}
\begin{proof}
    By hypothesis, $\zz_2\acts (X^2,\delta(X))$ and $\delta(X)$ is the set of fixed points of the action. Since $X$ is unbounded $X^2$ is not coarsely contained in $\delta(X)$.
    So, we can apply Lemma~\ref{coarsegysin} on $(X^2,\delta(X))$, and the claim follows.
\end{proof}
 
\begin{Remark}\label{aspecialcase}
If $X$ is unbounded and $\Hx[1](X^2-\delta(X))=0$, then it follows from Proposition~\ref{l:coarsegysin} that $\Hx[1]_{\zz_2}(X^2-\delta(X))=0$. As a consequence, we can write the beginning part of the coarse Gysin sequence \eqref{e:coarsegysin} as follows
 \begin{align}\label{coarsegysin2}
     0\rightarrow \zz_2 \rightarrow \Hx[2]_{\zz_2}(X^2-\delta(X)) \rightarrow \Hx[2](X^2-\delta(X)) \rightarrow \Hx[2]_{\zz_2}(X^2-\delta(X)) \rightarrow \cdots
 \end{align}
 This observation will be useful for us in our next theorem where we compute $\Hx_{\zz_2}(X^2-\delta(X))$ when $\Hx(X^2-\delta(X))$ is concentrated in some degree.
 \end{Remark}

 \begin{Theorem}\label{c:Conf(unif contr. space)}
      Suppose $X$ is a metric space such that for some $n\geq 1$
\[    
\Hx(X^2-\delta(X))=
\begin{cases}
\zz_2 &  *=n,\\
0 & \text{otherwise}.
\end{cases}
\]
       Moreover, suppose that $\Hx[i]_{\zz_2}(X^2-\delta(X))=0$ for some $i\geq n+1$.
      Then 
      \[\Hx_{\zz_2}(X^2-\delta(X)) = \left\{
        \begin{array}{ll}
            \zz_2 & \quad \text{if $n\geq 2$ and } 2\leq *\leq n,   \\
            0 & \quad \text{otherwise}.
        \end{array}
    \right.
\]
 \end{Theorem}
\begin{proof}
Elements of $\Hx[0](X^2-\delta(X))$ and $\Hx[0]_{\zz_2}(X^2-\delta(X))$ are constant functions on $X^2$ with support contained in a neighborhood of $\delta(X)$.
Since $\Hx[0](X^2-\delta(X))=0$, we have 
 $\Hx[0]_{\zz_2}(X^2-\delta(X))=0$.

Next, we will show that $\Hx_{\zz_2}(X^2-\delta(X))=0$ if $*\geq n+1$. Consider the following part of the coarse Gysin sequence where $*\geq 2$.
\begin{align}\label{mid} \rightarrow \Hx(X^2- \delta(X))\rightarrow \Hx_{\zz_2}(X^2-\delta(X))\rightarrow \Hx[*+1]_{\zz_2}(X^2-\delta(X)) \rightarrow \Hx[*+1](X^2- \delta(X))\rightarrow
 \end{align}
Since $\Hx(X^2-\delta(X))=0$ for $*\geq n+1$, the middle map is an isomorphism for $*\geq n+1$.  
 Hence, 
 \[\Hx[*]_{\zz_2}(X^2-\delta(X))=\Hx[n+1]_{\zz_2}(X^2-\delta(X)) \quad \text{ for all } *\geq n+1.
 \]
 That means, if  $\Hx[n+1]_{\zz_2}(X^2-\delta(X))\neq 0$ then $\Hx_{\zz_2}(X^2-\delta(X))\neq 0$ for all $*\geq n+1$. 
 By hypothesis, $\Hx[i]_{\zz_2}(X^2-\delta(X))=0$ for some $i\geq n+1$.
Therefore, $\Hx[n+1]_{\zz_2}(X^2-\delta(X))=0$ and consequently, 
\[\Hx_{\zz_2}(X^2-\delta(X))=0 \quad \text{ for all } *\geq n+1\]
We divide the rest of the calculations into three cases. 

\hfill

\textbf{Case 1 ($n=1$):}
We only have to show that $\Hx[1]_{\zz_2}(X^2-\delta(X))=0$.
By the hypothesis, $\Hx[1](X^2-\delta(X))=\zz_2$ and
$\Hx[2]_{\zz_2}(X^2-\delta(X))= 0$.
Hence, we get the following from the coarse Gysin sequence
\begin{align}\label{beginning}
    0  \rightarrow \Hx[1]_{\zz_2}(X^2-\delta(X))  \rightarrow \zz_2\rightarrow \Hx[1]_{\zz_2}(X^2-\delta(X))\oplus \zz_2 \rightarrow 0\nonumber
\end{align}
Since the third map is surjective, $\Hx[1]_{\zz_2}(X^2-\delta(X))$ cannot have more than one element and hence it is trivial.

\hfill

\textbf{Case 2 ($n=2$):} By hypothesis, $\Hx[1](X^2-\delta(X))=0$, Therefore, $\Hx[1]_{\zz_2}(X^2-\delta(X))=0$ (see remark~\ref{aspecialcase}).
Furthermore, we have $\Hx[2](X^2-\delta(X))=\zz_2$ by hypothesis and we already showed $\Hx[3]_{\zz_2}(X^2-\delta(X))=0$.
 So a part of the sequence (\ref{coarsegysin2}) takes the following form.
 \[0\rightarrow \zz_2\rightarrow \Hx[2]_{\zz_2}(X^2-\delta(X))\rightarrow \zz_2\rightarrow \Hx[2]_{\zz_2}(X^2-\delta(X))\rightarrow 0
 \]
 So there is an injective map and a surjective map from $\zz_2$ into 
 $\Hx[2]_{\zz_2}(X^2-\delta(X))$.
 It follows that  $\Hx[2]_{\zz_2}(X^2-\delta(X))=\zz_2$.

 \hfill
 
 \textbf{Case 3 ($n>2$):} In this case, $\Hx[1](X^2-\delta(X))=0$.
 The remark \ref{aspecialcase} gives us $\Hx[1]_{\zz_2}(X^2-\delta(X))=0$.

  Since $\Hx[2](X^2-\delta(X))=0$, the beginning part of the sequence (\ref{coarsegysin2}) takes the following form.
  \[0\rightarrow \zz_2\rightarrow \Hx[2]_{\zz_2}(X^2-\delta(X))\rightarrow 0\rightarrow \cdots
  \]
  Hence $\Hx[2]_{\zz_2}(X^2-\delta(X))=\zz_2$.

Since $\Hx(X^2-\delta(X))= 0$ for $* \leq n-1$,
it follows from $(\ref{mid})$ that
  \[\Hx_{\zz_2}(X^2-\delta(X))= \Hx[*+1]_{\zz_2}(X^2-\delta(X)) \quad \text{ when } \quad 2\leq * \leq n-2.
  \]
  
  That implies  \[\Hx_{\zz_2}(X^2-\delta(X))=\Hx[2]_{\zz_2}(X^2-\delta(X))=\zz_2 \text{ when } 2\leq * \leq n-1.
  \]

Finally to compute  $\Hx[n]_{\zz_2}(X^2-\delta(X))$, consider the following part of the coarse Gysin sequence (Lemma~\ref{coarsegysin}).
\[\dots \rightarrow \Hx[n]_{\zz_2}(X^2-\delta(X)) \rightarrow \Hx[n](X^2-\delta(X))\rightarrow \Hx[n]_{\zz_2}(X^2-\delta(X))\rightarrow\Hx[n+1]_{\zz_2}(X^2-\delta(X))\rightarrow\cdots 
 \]
 Since $\Hx[n](X^2-\delta(X))\neq 0$ by assumption, it follows from the above sequence that $\Hx[n]_{\zz_2}(X^2-\delta(X))\neq 0$.
 We already know that the fourth term is trivial in the above sequence. That means the third map is surjective. Since $\Hx[n](X^2-\delta(X))=\zz_2$, we can conclude that $\Hx[n]_{\zz_2}(X^2-\delta(X))=\zz_2$.
 \end{proof}

 \begin{Remark}\label{r:lower bound cobdim}
     It follows from the first part of the proof of the above theorem that  if $\Hx(X^2-\delta(X))=0$ for all $*\geq n+1$ and $\Hx_{\zz_2}(X^2-\delta(X))=0$ for some $*\geq n+1$, then $\Hx_{\zz_2}(X^2-\delta(X))=0$ for all $*\geq n+1$.
     In particular, for such $X$, $cvk^*(X)$ is trivial for $*\geq n+1$ and hence $\cobdim(X)\leq n$.
 \end{Remark} 
 \section{Upper bound of cobdim}\label{s:ub}
 
In this section, our goal is to apply Theorem~\ref{c:Conf(unif contr. space)} to estimate $\cobdim$ for proper, uniformly acyclic $n$-manifolds.
The key tool that we are going to exploit here is  a coarse Alexander duality theorem that holds for such spaces. 
 In fact, such duality holds for a more general class of metric spaces called coarse PD($n$) spaces, first introduced by 
Kapovich and Kleiner in~\cite{KK} where they  proved a coarse Alexander duality theorem for these spaces.
A different treatment of  coarse PD($n$) space and coarse Alexander duality using coarse cohomology is given in \cite{BB20}.
Let us now recall the definition of a coarse PD($n$) space from~\cite{BB20}.

\begin{Definition}\label{def pdn}
    A metric space $X$ is a \emph{coarse PD($n$) space}, if there exist chain maps $p: \C(X;\zz) \to \cx[n-*](X;\zz)$ and $q: \cx[n-*](X;\zz) \to \C(X;\zz)$, so that $pq$ and $qp$ are chain homotopic to identities via chain homotopies $G:\cx(X;\zz) \to \cx[*+1](X;\zz)$ and $F: \C(X;\zz) \to \C[*-1](X;\zz)$ which are controlled:
    \begin{align*}
        \forall \phi\in \C(X;\zz) &\qquad \Supp{ p(\phi)} \csubset \Supp{\phi}, \\
        \forall \phi\in \C(X;\zz) &\quad \Supp{ F(\phi)} \ccap \gD \csubset \Supp{\phi} , \\
        \forall c \in \cx(X;\zz) &\quad \Supp{q(c)} \ccap \gD \csubset \Supp{c}, \\
        \forall c\in \cx(X;\zz) &\qquad \Supp{ G(c)} \csubset \Supp{c}.
    \end{align*}
\end{Definition} 

\begin{Example}
    Any proper, uniformly acyclic $n$-manifold is a coarse PD($n$) space~\cite{BB20}*{Corollary 8.3}. In particular, the universal cover of a closed, aspherical $n$-manifold is a coarse PD($n$) space.
\end{Example}
  We now recall  the coarse Alexander duality theorem from~\cite{BB20}.
\begin{Theorem}[Coarse Alexander duality \cite{BB20}]\label{CAD}
 If $X$ is a coarse PD($n$) space, then for any $A\subset X$ and finitely generated abelian group $G$
 \[\Hx(X-A;G)\cong\hx[n-*](A;G).\]
\end{Theorem}
As a consequence we have the following.
\begin{Lemma}\label{concentrated}
    Let $X$ be a coarse PD($n$) space. Then 
\[\hx(X;\zz_2)=
\begin{cases}
\zz_2 & *=n,\\
0 & \text{otherwise}.
\end{cases}
\]
    \end{Lemma}

    \begin{proof}
 By Theorem~\ref{CAD}, we have $\hx[n-*](X;G)=\Hx(X-X;G)$.
Observe that $\Cx(X-X;\zz_2)=\C(X;\zz_2)$. Since $\C(X;\zz_2)$ is acyclic by Lemma~\ref{l:C is acyclic}, we obtain
\[\hx(X;\zz_2)=\Hx[n-*](X-X;\zz_2)=
\begin{cases}
\zz_2 & *=n,\\
0 & \text{otherwise}.
\end{cases}
\]
\end{proof}

 For the rest of the paper, the omitted coefficient will mean $\zz_2$.
\begin{Lemma}\label{ordered conf for pD(n)}
If $X$ is a coarse PD($n$) space, then 

\[    
\Hx(X^2-\delta(X))=
\begin{cases}
\zz_2 &  *=n, \\
0 & \text{otherwise}.
\end{cases}
\]
\end{Lemma}
\begin{proof}
If $X$ is a coarse PD($n$) space,
then $X^2$ is a coarse PD($2n$) space. 
We obtain
\begin{align*}
    \Hx(X^2 -\delta(X))&= \hx[2n-*](\delta(X)) &&\text{by Theorem~\ref{CAD}}\\
    &=\begin{cases}
\zz_2 & * = n,\\
0 & \text{otherwise}.
\end{cases}
&&\text{by Lemma~\ref{concentrated}}
\end{align*}

\end{proof}

\begin{Lemma}\label{l:equiv is zero for large degree}
    If $X$ is a proper, uniformly acyclic $n$-manifold where $n\geq 1$, then $\Hx[i]_{\zz_2}(X^2-\delta(X))=0$ for $i\geq 2n+2$.
\end{Lemma}
\begin{proof}
If $X$ is bounded, then $X^2\ceq \delta(X)$ and the claim follows from Theorem~\ref{c:equicoa=coaquo}\eqref{i:equicoa=coaquo1}.
So, we assume that $X$ is unbounded.
This implies that $X^2$ is not coarsely contained in $\delta(X)$.
By hypothesis, $X^2$ is uniformly acyclic and locally acyclic.
Hence we can apply Corollary~\ref{c:equicoa=coaquo} to the pair $(X^2,\delta(X))$ and obtain
\[\Hx_{\zz_2}(X^2-\delta(X))=\varinjlim \rH[*-1]((X^2-N_r(\delta(X))/\zz_2).\]
Since $(X^2-N_r(\delta(X)))/\zz_2$ is a $2n$-manifold, $\Hx_{\zz_2}(X^2-\delta(X))=0$ for all $*\geq 2n+2$.
\end{proof}

\begin{Theorem}\label{w/o bdry}
    If $X$ is a proper, uniformly acyclic $n$-manifold, then $\cobdim(X)\leq n$. 
\end{Theorem}
\begin{proof}
    For $n=0$, the claim is trivial.
    For $n\geq1$, Lemma~\ref{ordered conf for pD(n)} and Lemma~\ref{l:equiv is zero for large degree} implies that $X$ satisfies the hypothesis of Theorem~\ref{c:Conf(unif contr. space)}. 
    Hence by Theorem~\ref{c:Conf(unif contr. space)}, we have $\Hx_{\zz_2}(X^2-\delta(X))=0$ for $*\geq n+1$.
    This implies $\cobdim(X)\leq n$.
\end{proof}

\begin{Remark}
    Note that to prove Theorem~\ref{w/o bdry}, we did not need the full strength of the Theorem~\ref{c:Conf(unif contr. space)}.
    We needed to show $\Hx_{\zz_2}(X^2-\delta(X))=0$ for $*\geq n+1$ which only requires vanishing of $\Hx(X^2-\delta(X))$ for $*\geq n+1$ and vanishing of $\Hx_{\zz_2}(X^2-\delta(X))$ for some $*\geq n+1$ (see remark~\ref{r:lower bound cobdim}).
\end{Remark}
Our next goal is to improve Theorem~\ref{w/o bdry} to include manifolds with boundaries. 
For that, we need to impose a condition on the metric of the boundary. 
This is the purpose of the following definition which is inspired by the uniformly locally $k$-connected space defined in~\cite{DFW}.

\begin{Definition}\label{unlcontr}
  A metric space $(X,d)$ is uniformly locally acyclic, if
for every  $\epsilon >0$, there is a $\delta > 0$ such that any ball of radius $\delta$ is acyclic inside a ball of radius $\epsilon$.
\end{Definition}

\begin{Example}
    Any compact, locally acyclic space is uniformly locally acyclic. Similarly, any locally acyclic space that admits a cocompact group action by homeomorphisms, is uniformly locally acyclic.
    In particular, universal cover of a compact manifold is uniformly locally acyclic.

Any uniformly locally acyclic space is also locally acyclic. However, the converse is not true. For example, the set $\{\frac{1}{n}\}$ with the subspace metric from $\rr$ is locally acyclic, but it is not uniformly locally acyclic.
\end{Example}

\begin{Lemma}\label{l:local contr gives unif contr}
Let $(X,d)$ be a uniformly locally 
acyclic metric space. Then the space $X\times [1,\infty)$ has a metric that makes the space uniformly acyclic away from $X\times \{1\}$ and the map $x\mapsto (x,1)$ is an isometric embedding of $X$ into $X\times [1,\infty)$.
\end{Lemma}
\begin{proof}
 The construction of the metric follows the one appearing in Lemma 2.2 of~\cite{DFW}.
Choose a continuous strictly increasing function $\phi: [1,\infty) \rightarrow [1,\infty)$ with
$\phi(1) = 1$. Let $d$ be the original metric on $X$ and define a function $\rho'$ by
\begin{enumerate}
    \item $\rho'((x, t),(x', t)) = \phi(t)d(x, x')$.
\item $\rho'((x, t),(x, t')) = |t-t'|$.
\end{enumerate}
We then define $\rho : (X \times [1,\infty))^2 \rightarrow [0,\infty)$ to be
\[\rho((x, t),(x', t')) = \inf \sum_{i=1}^{l}\rho'((x_i, t_i),(x_{i-1}, t_{i-1}))\]
where the sum is over all chains
\[(x, t)=(x_0, t_0),(x_1, t_1),\ldots,(x_l, t_l)=(x', t')\]
and each segment is either horizontal or vertical. 
Also $\phi(1)=1$ implies that $X\times\{1\}$ with the subspace metric is isometric to $X$ via the map $(x,1)\mapsto x$.
Now we will describe a $\phi$, so that the corresponding metric $\rho$ makes $X\times [1,\infty)$ uniformly acyclic away from $X\times \{1\}$.
Since $X$ is uniformly locally acyclic, we have an infinite positive decreasing sequence $\{r_i\}$ with $r_1=1$ such that for every $x\in X$, the inclusions $\dots\subset B_{r_{i}}^d(x)\subset B_{r_{i-1}}^d(x)\subset$ are nullhomotopic maps.
Set $\phi(t)=\frac{1}{r_{t}}$ for $t\in \nn$.
For nonintegral values of $t$, we set
\[\phi(t)=\phi([t])+(t-[t])\phi([t]+1)
\]
Suppose, $N_i=\frac{\phi(i)}{\phi(i-1)}$.
Now we consider the ball $B^\rho_k(x,i)\subset X\times [1,\infty)$.
Note that $B^\rho_k(x,i)\subset B^d_{\frac{k}{N_{i-k}}}(x)\times[i-k,i+k]$ and that $B^\rho_k(x,i)$ contracts in itself to $B^\rho_k(x,i)\cap (X\times[i-k,i])\subset B^d_{\frac{k}{N_{i-k}}}(x)\times [i-k,i]$.
Also, $B^d_{\frac{k}{N_{i-k-1}}}(x)\times \{i-k-1\}\subset B_{k+2}^\rho(x,i)$.
So, $B^\rho_k(x,i)$ can be contracted inside $B_{k+2}^\rho(x,i)$ by pushing it down to $(i-k-1)$-level and contracting it there.
\end{proof}

The following gluing lemma in a slightly different form can be found in \cite{BH}.
\begin{Lemma}[\cite{BH}, Lemma I.5.24]\label{gluing}
Let $X_1$ and $X_2$ be two proper metric spaces. Let $A_i\subset X_i$ be the closed subsets and $f:A_1\rightarrow A_2$ be an isometry. Let $Y$ be the space obtained by gluing $(X_i,A_i)$ along $A_i$ via the map $f$. 
Define $d:Y\times Y\rightarrow \rr$ as follows
\[d(x,y)=\begin{cases}
d_i(x,y) & \text{if $x,y\in X_i$}\\
\inf_{a\in A_1}\{d_1(x,a)+d_2(f(a),y)\} & \text{if $x\in X_1, y\in X_2$}.
\end{cases}
\]
Then, 
\begin{enumerate}
    \item $d$ is a proper metric on $Y$.
    \item The canonical inclusions $X_i\hookrightarrow Y$ are isometric embedding.

\end{enumerate}
\end{Lemma}

\begin{Prop}\label{p:w boundary embed in without boundary}
Any uniformly acyclic, proper $n$-manifold with uniformly locally acyclic boundary admits an isometric embedding into  a uniformly acyclic, proper $n$-manifold. 
\end{Prop}
\begin{proof}
Since $\d M$ is uniformly locally acyclic, Lemma \ref{l:local contr gives unif contr} allows us to equip $\partial M \times [1,\infty)$ with a metric so that it is uniformly acyclic away from $\partial M\times \{1\}$.
Let $\rho,\mu:\rr_+\rightarrow \rr_+$ be two functions such that any ball $B(x,r)$ in  $\partial M\times [1,\infty)$ is acyclic inside $B(x,\rho(r))$ whenever $d(x,\partial M \times \{1\})\geq \mu(r)$.
We glue $\partial M \times [1,\infty)$ to $M$ along $\d M$ by the attaching map $(x,1)\mapsto x$.
Let $Y$ be  the resulting  space.
By Lemma~\ref{gluing}, there is a proper metric on $Y$ such that the canonical maps $M\hookrightarrow Y$ and $\partial M\times [1,\infty)\hookrightarrow Y$ are isometric embedding.
For the rest of the proof we will regard $M$ and $\partial M\times [1,\infty)$ as subspaces of $Y$.

We claim that $Y$ is uniformly acyclic.
Since $M$ is uniformly acyclic, there exists a function $\tau:[0,\infty)\rightarrow [0,\infty)$ such that, for any $r\geq 0$, any ball of radius $r$ in $M$ is acyclic inside a concentric ball of radius $\tau(r)$.
Take a ball $B_r(x)$ of radius $r$ in $Y$.
If $x\in M$ and $d(x,\partial M)\geq r$, then $B_r(x)$ is contained inside $M$ because points in $\partial M\times [1,\infty)$ are at least as far from $x$ as points in $\partial M$ by the construction of the metric on $Y$. Hence, $B_r(x)$ is acyclic inside $B_{\tau(r)}(x)$.
If $x\in \partial M\times [1,\infty)$ and $d(x, \partial M)\geq \mu(r)+r$, then $B(x,r)\subset \partial M\times [1,\infty)$ and hence is acyclic in $B(x,\rho(r))$.
If $d(x,\d M)\leq \mu(r)+r$, we can deformation retract $B_r(x)\cap (\partial M\times [1,\infty))$ inside $\d M\times [1,\infty)$ by sliding it along $[1,\infty)$ until it lands on $\partial M\times \{1\}$.
Note that this deformation takes place inside $\d M\times [1,\mu(r)+2r]$, and the diameter shrinks as one approaches $\d M\times \{1\}$.
Hence the deformation takes place in a set of diameter at most  $\mu(r)+2r+2r$.
Also, the deformed ball is now contained in $M$ and has diameter at most $2r$ and hence it is acyclic inside a set of diameter at most $\tau(2r)$ by uniform acyclicity of $M$.
Hence, we conclude that $B_r(x)$ is acyclic inside a set of diameter at most $\mu(r)+4r+\tau(2r)$.
Hence $Y$ is uniformly acyclic.
\end{proof}

As a consequence of the above proposition, we get the following.

\begin{Theorem}\label{cobdim of manifold}
If $X$ is a proper, uniformly acyclic $n$-manifold with uniformly locally acyclic boundary, then $\cobdim(X)\leq n$.
\end{Theorem}
\begin{proof}
By Proposition~\ref{p:w boundary embed in without boundary}, we have a uniformly acyclic proper $n$-manifold $Y$ such that $X$ embeds isometrically in $Y$.
By Theorem~\ref{t:cobdim increases}, it follows that $\cobdim(X)\leq \cobdim(Y)$.
By Corollary~\ref{w/o bdry}, we know $\cobdim(Y)\leq n$ and hence $\cobdim(X)\leq n$.
\end{proof}
Now we can state the following improvement of Corollary~\ref{c:obstruction rn}. The proof is immediate from Theorem~\ref{t:cobdim increases} and Theorem~\ref{cobdim of manifold}.

\begin{Corollary}\label{c:cobdim geq emb}
     If $\cobdim(X)\geq n$, then $X$ cannot be coarsely embedded into a proper, uniformly  acyclic $(n-1)$-manifold with a uniformly locally acyclic boundary.
\end{Corollary}

\begin{Definition}[Cocompact action dimension]
  The $\textit{cocompact action dimension}$ $\cadim(G)$ of a group $G$ is the least dimension of a contractible manifold (possibly with boundary)  that admits a proper cocompact $G$-action.
 \end{Definition}
 
 \begin{Corollary}\label{cadim>cobdim}
      $\cadim(G)\geq \cobdim(G)$.
 \end{Corollary}
 \begin{proof}
 Suppose $\cadim(G)=n$.
 Then $G$ admits  a proper, cocompact action on an acyclic $n$-manifold $M$.
 Choose a point $x_0\in M$.
 By Milnor--Schwarz Lemma, the map $g\mapsto g.x_0$ gives a coarse equivalence $f:G\rightarrow M$.
 Since $M$ is acyclic and it admits a cocompact action, $M$ is uniformly acyclic.
 Similarly, since $\partial M$ is locally acyclic and it admits a cocompact action, it is uniformly locally acyclic.
 Since $M$ is proper,
 by Corollary~\ref{c:cobdim geq emb}, we get $\cobdim(G)\leq n=\cadim(G)$.
 
 \end{proof}

 \begin{Remark}
     The action dimension $\actdim(G)$ of a group $G$ is the least dimension of a contractible manifold  that admits a proper $G$-action.
     Note that $\cadim(G)\geq \actdim(G)$, however, we do not know of  any groups where $\cadim>\actdim$.
     Nonetheless, we believe that $ \cobdim(G)$ gives a lower bound to $\actdim(G)$.
     A naive approach to prove this might be as follows:  
     Suppose, $G$ admits a proper action on a contractible $n$-manifold $M$.
    Then the map $f:G\rightarrow M$, sending $G$ to one of its orbit gives a coarse embedding.
    Furthermore, the image of $f$ is uniformly contractible in $M$: 
    for any $r$ there exists $s$ such that any ball of radius $r$ in $M$ centered at a point in $f(G)$ is uniformly contractible inside a concentric ball of radius $s$. 
    This should imply that there is a bounded function $p:f(G)\rightarrow (0,\infty)$ such that the space $X=\cup_{x\in f(G)} N_{p(x)}(x)$ is a uniformly contractible manifold with boundary. 
     Since $G$ acts on $X$ cocompactly, $X$ has uniformly locally acyclic boundary.
    Theorem~\ref{cobdim of manifold} then implies that $\cobdim(X)\leq n$.
    Since $G$ coarsely embed into $X$ by the map $f$, we get $\cobdim(G)\leq \cobdim(X)\leq n$, proving the desired claim.
    However, the author does not know how to prove that there exists such $X$.
    \end{Remark}
 \begin{Question}
     Is $\cobdim(G)\geq \actdim(G)$?
 \end{Question}
 \section{Lower bound of cobdim}\label{s:lb}
\begin{Theorem}\label{t:cobdim of unifcontr}
 If $\Hx(X^2-\delta(X))=0$ for $*\leq n-1$, then $\cobdim(X)\geq n$.
\end{Theorem}

\begin{proof}
For $n=0$, the claim is trivial. If $n=1$, the assumption says $\Hx[0](X^2-\delta(X))=0$. This means $X$ is unbounded, otherwise, non zero constant functions from $X^2$ to the $\zz_2$ give nontrivial elements in $\Hx[0](X^2-\delta(X))$. Hence, $\cobdim(X)\geq 1$ in this case.

Suppose $n\geq 2$ and $f:X\rightarrow \ell^\infty$ is an isometry. 
We first show that $f^*:\Hx[2]_{\zz_2}(\ell^\infty)^2-\delta(\ell^\infty))\xrightarrow{f^*} \Hx[2]_{\zz_2}(X^2-\delta(X)) $ is a nontrivial map.
To see that, consider the following part of the maps between the concerned coarse Gysin sequences. 
Our goal is to show that the second vertical map is nontrivial.

\[\begin{tikzcd}[column sep= small]
\arrow{r} \Hx[1]((\ell^\infty)^2-\delta(\ell^\infty))\arrow{r} & \Hx[1]_{\zz_2}((\ell^\infty)^2-\delta(\ell^\infty))\oplus \zz_2 \arrow{d}{f^*} \arrow{r} &  \Hx[2]_{\zz_2}((\ell^\infty)^2-\delta(\ell^\infty))\arrow{d}{f^*} \arrow{r}  & {}\dots \\
  \arrow{r} \Hx[1](X^2-\delta(X)) \arrow{r} & \Hx[1]_{\zz_2}(X^2-\delta(X))\oplus \zz_2 \arrow{r}{j} & \Hx[2]_{\zz_2}(X^2-\delta(X))  \arrow{r} & {}  \ldots
\end{tikzcd}
\]

By the commutativity of the diagram, our claim follows if we can show that $j$ is injective and the first vertical map is nontrivial.

Since $\Hx[1](X^2-\delta(X))=0$, it follows from the long exact sequence (\ref{coarsegysin2}) that $j$ is injective. 

Next, we show that the first vertical map in the above commutative diagram is non trivial.
It is equivalent to showing that the following map is nontrivial.
\[f^*: \Hx[1]_{\zz_2}((\ell^\infty)^2-\delta(\ell^\infty),\delta(\ell^\infty))\rightarrow  \Hx[1]_{\zz_2}((X^2-\delta(X),\delta(X))\]
We can show that using the following  diagram.
\[\begin{tikzcd}[column sep= small]
 & 0 \arrow{d}{f^*}\arrow{r} & \H[0](\C(\delta(\ell^\infty))\arrow{d}{f^*} \arrow{r} & \Hx[1]_{\zz_2}((\ell^\infty)^2-\delta(\ell^\infty),\delta(\ell^\infty)) \arrow{d}{f^*} \arrow{r} &  \Hx[1]((\ell^\infty)^2-\delta(\ell^\infty))\arrow{d}{f^*} \arrow{r}  & 0 \arrow{d}{f^*} \\
  & 0 \arrow{r} &\H[0](\C(\delta(X))) \arrow{r}  & \Hx[1]_{\zz_2}(X^2-\delta(X),\delta(X)) \arrow{r}{j} & \Hx[1](X^2-\delta(X))  \arrow{r} & 0
\end{tikzcd}
\]
The long exact sequence is due to the second long exact sequence of Lemma~\ref{singular Gysin} combined with the fact that $\Hx[1]((\ell^\infty)^2-\delta(\ell^\infty)), \H[1](\C(\delta(\ell^\infty))), \Hx[1](X^2-\delta(X))$ and $\H[1](\C(\delta(X)))$ all are trivial.
The second vertical map is an isomorphism as both the domain and the range are constant functions defined on the respective spaces.
The fourth vertical map is an isomorphism because both the domain and the range are trivial.
So, applying the Five lemma on the above diagram, we conclude that the middle vertical map is an isomorphism, as desired.

Hence we obtain that the map $f^*:\Hx[2]_{\zz_2}(\ell^\infty)^2-\delta(\ell^\infty))\rightarrow \Hx[2]_{\zz_2}(X^2-\delta(X)) $ is nontrivial.

Let us now consider the maps between the following parts of the  coarse Gysin sequences where $*\geq 2$.
\[\begin{tikzcd}[column sep= small]
\arrow[r]& \Hx((\ell^\infty)^2-\delta(\ell^\infty))\arrow{d}{f^*}\arrow{r} & \Hx_{\zz_2}(\ell^\infty)^2-\delta(\ell^\infty)) \arrow{d}{f^*} \arrow{r} &  \Hx[*+1]_{\zz_2}((\ell^\infty)^2-\delta(\ell^\infty))\arrow{d}{f^*} \arrow{r} & {} \\ 
\arrow[r] & \Hx(X^2-\delta(X))  \arrow{r} &  \Hx_{\zz_2}(X^2-\delta(X))  \arrow{r} & \Hx[*+1]_{\zz_2}(X^2-\delta(X))\arrow{r} & {} 
\end{tikzcd}
\]

Note that the first terms of both sequences above are trivial for $*\leq n-1$.
That implies that the third horizontal maps in the above diagram are injective. Hence by commutativity of the diagram, if the second vertical map is nontrivial then so is the third vertical map when $*\leq n-1$.
We saw previously that, when $*=2$ the second vertical map is injective. 
It now follows by induction that the third vertical map is injective when $*\leq n-1$.
In particular, when $*=n-1$, injectivity of the third vertical map means $cvk^n(X)\neq 0$.
Hence, $\cobdim(X)\geq n$.

\end{proof}
As a consequence we get the following, which also follows from the work of Yoon~\cite{Yoon}*{Lemma 5.3}.
\begin{Corollary}\label{lb of cob}
      If $X$ is a coarse $PD(n)$ space, then $X$ does not coarsely embed into a proper, uniformly acyclic $(n-1)$-manifold with a uniformly locally acyclic boundary.
\end{Corollary}
\begin{proof}
If $X$ is a coarse $PD(n)$ space, then it satisfies the hypotheses of Theorem~\ref{t:cobdim of unifcontr}. Hence, $\cobdim(X)\geq n$.
The claim now follows from Corollary~\ref{c:cobdim geq emb}.
\end{proof}

\begin{Example}\label{cobdimof manifold}
If $X$ is a proper, uniformly acyclic $n$-manifold, then it satisfies all the hypotheses of the above corollary.
Hence $\cobdim(X)\geq n$. Theorem \ref{cobdim of manifold} implies that  $\cobdim(X)\leq n$. Hence $\cobdim(X)=n$ whenever $X$ is a proper, uniformly acyclic $n$-manifold.
\end{Example}

\end{document}